\documentclass[a4paper,reqno,11pt]{amsart}
\usepackage{geometry}
\usepackage{mathrsfs}
\usepackage{amsmath} \numberwithin{equation}{section}
\usepackage{amsthm}
\usepackage{amsfonts}
\usepackage{amssymb}
\usepackage{cases}
\usepackage{enumerate}
\usepackage[english]{babel}
\usepackage{graphics}

\theoremstyle{plain}
\newtheorem{theorem}{Theorem}[section]
\newtheorem{lemma}[theorem]{Lemma}
\newtheorem{corollary}[theorem]{Corollary}

\theoremstyle{definition}
\newtheorem{definition}[theorem]{Definition}
\newtheorem{example}[theorem]{Example}
\newtheorem{remark}[theorem]{Remark}

\newcommand\C{\mathbb C}         
\newcommand\R{\mathbb R}         
\newcommand\Z{\mathbb Z}         
\newcommand\N{\mathbb N}         
\newcommand{\uhp}{{\mathbb H}}
\newcommand{\Nev}{{\mathcal N}}
\newcommand{\G}{{\mathscr G}}      
\newcommand\D{\mathbb D}  
\newcommand\T{\mathbb T}

\renewcommand\Im{\operatorname{Im}}
\renewcommand\Re{\operatorname{Re}}

\newcommand{\cP}{{\rm\bf P}}  
\newcommand{\aand}{{\quad\text{and}\quad}}

\newcommand{\eps}{\varepsilon}

\newcommand{\sphere}{{\widehat{\mathbb C}}}

\usepackage[dvipdfmx]{graphicx}

\title{Nonlinear resolvents and decreasing Loewner chains}
\author[I. Hotta]{Ikkei Hotta}
\author[S. Schlei{\ss}inger]{Sebastian Schlei{\ss}inger}
\author[T. Sugawa]{Toshiyuki Sugawa}
\thanks{The first author was supported by JSPS KAKENHI Grant no.20K03632}
\thanks{The second author was supported by the German Research Foundation (DFG), project no.401281084}
\thanks{The third author was supported by JSPS KAKENHI Grant no.JP17H02847}
\thanks{This version of the article has been accepted for publication, and is subject to Springer Nature's AM terms of use, but is not the Version of Record and does not reflect post-acceptance improvements, or any corrections. The Version of Record is available online at: http://dx.doi.org/ 10.1007/s12220-023-01544-y}

\begin{document}
\parindent 0pt

\maketitle

\begin{center}\textit{Dedicated  to  the  memory  of  Gabriela Kohr.}\end{center}

\vspace{5mm}

 \begin{abstract}
In this article we prove that nonlinear resolvents of infinitesimal generators on bounded and
convex subdomains of $\C^n$ are decreasing Loewner chains. Furthermore, we consider the problem of the existence of nonlinear resolvents on unbounded convex domains in $\C$. In the case of the upper half-plane, we obtain a complete solution by using that nonlinear resolvents of certain generators correspond to semigroups of probability measures with respect to free convolution.
 \end{abstract}

 \tableofcontents

\section{Introduction}

Let $D\subset \C^n$ be a complete (Kobayashi) hyperbolic domain. The (decreasing) Loewner partial differential equation on $D$ has the form
\begin{equation}\label{intro_1_eq}
 \frac{\partial}{\partial t}f_t(z)=f'_t(z)\cdot G(t,z)\quad \text{for a.e.\ $t\geq 0$}, \quad f_0(z)=z\in D,
\end{equation}
where a mapping $f_t(z):D\to D$ is locally absolutely continuous in $t\ge0$ and
holomorphic in $z\in D$ and
$f'_t(z)$ denotes the holomorphic Jacobian matrix of $f_t(z)$ with respect to $z$ and $G(t,z)$ is a Herglotz vector field on $D$, i.e.\ $z\mapsto G(t,z)$ is an infinitesimal generator of a semigroup on $D$
for almost every $t\geq0$ and $(t,z)\mapsto G(t,z)$ satisfies certain regularity conditions.\\

 Equation \eqref{intro_1} has a unique solution of univalent mappings $f_t:D \to D$, which form a decreasing Loewner chain, i.e.\ $f_t(D)\subset f_s(D)$ whenever $s\leq t$, $f_0$ is the identity, and $t\mapsto f_t$ is continuous with respect to locally uniform convergence. This result is well-known, but as the literature usually focuses on increasing Loewner chains, we also provide a proof (Theorem \ref{thm_1}).

Now let us replace $G(t,z)$ by $G(t,f_t(z))$ to obtain the following modified equation:
\begin{equation}\label{intro_2_eq}
 \frac{\partial}{\partial t}f_t(z)=f'_t(z)\cdot G(t,f_t(z))\quad \text{for a.e.\ $t\geq 0$}, \quad f_0(z)=z\in D.
\end{equation}

At first sight, it seems unnatural to look at this equation, for, if $G$ is a generator on $D$ and $\varphi:D\to D$ is a holomorphic self-mapping, then $G\circ \varphi$ need not be a generator.
On the other hand, equation \eqref{intro_2_eq} is simpler than \eqref{intro_1_eq} in the following sense. 
While the inverse functions $g_t=f_t^{-1}$ of the solution to \eqref{intro_1_eq} satisfy Loewner's ordinary differential equation 
\begin{equation*}
\frac{\partial}{\partial t}g_t(z) = - G(t,g_t(z)),
\end{equation*}
 the inverse functions $g_t$ of the solution to \eqref{intro_2_eq}, if it exists, 
satisfy \[\frac{\partial}{\partial t}g_t(z) = - G(t,z),\quad \text{i.e.} \quad  g_t(z)=z-\int_0^t G(s,z)ds.\]

Note that we can define $g_t$ by this equation on the whole domain $D$.
Special cases of \eqref{intro_2_eq} appear in different contexts: theory of continuous semigroups of holomorphic functions (\cite{ESS20} and \cite{GKH}), limit of the radial/chordal multiple SLEs (\cite{dMS16}, \cite{dMHS18}, \cite{HS}) and semigroups of probability measures with respect to free convolution (\cite{HS}).
In this article we explain these relations in a more general setting.\\

The initial value problem \eqref{intro_1_eq} has a unique solution, which consists of univalent functions forming a decreasing Loewner chain (see Theorem \ref{thm_1}).
In the simplest case, $G(t,z)$ does not depend on $t$, i.e.\ $G(t,z)=G(z)$ for an infinitesimal generator $G$ on $D$.
Let $s\geq0$. Then $(f_{t+s}(z))_{t\geq0}$ and $((f_s\circ f_t)(z))_{t\geq0}$ both satisfy \eqref{intro_1_eq} with initial value $f_s(z)$ at $t=0$. Applying $f_s^{-1}$ yields that $(f_s^{-1}(f_{t+s}(z)))_{t\geq0}$ and $(f_t(z))_{t\geq0}$ both satisfy \eqref{intro_1_eq} with initial value $z$ at $t=0$. So $f_{t+s}=f_s\circ f_t$ in this case. As $\left.\frac{\partial f_t(z)}{\partial t}\right\vert_{t=0} = G(z)$, the solution to \eqref{intro_1_eq} is then simply the semigroup associated to $G$.
 (We note that some authors call $-G$ the infinitesimal generator of the semigroup.)\\

By approximating the general version of \eqref{intro_1_eq} by a piecewise constant Herglotz vector field, we can think of the solution as an infinite composition of 
semigroup mappings which are infinitesimally close to the identity.

Equation \eqref{intro_2_eq}, on the other hand, is related to the nonlinear resolvents of $G$. 
Assume that $D$ is bounded and convex. It is well-known that the nonlinear resolvent equation (at time $t$)
\[w= z - tG(z)\] has a solution $z=f_t(w)$ for every $w\in D$ and every $t\geq 0$ so that $f_t$ is a univalent holomorphic self-map of $D.$ See Section \ref{sec_semi} for details.

\begin{theorem}\label{intro_1} Let $D\subset \C^n$ be a domain which is bounded and convex. 
 Let $G$ be an infinitesimal generator on $D$.
Then there exists a unique solution $(f_t)$ to
\begin{equation}\label{auto} \frac{\partial}{\partial t}f_t(z)=f'_t(z)\cdot G(f_t(z)) \quad \text{for all $t\geq 0$}, \quad f_0(z)=z\in D.
\end{equation}
The mappings $f_t$ are the nonlinear resolvents of $G$ and they form a decreasing Loewner chain.
\end{theorem}

\begin{definition}\label{def_dif}Let $D\subset \C^n$ be a complete hyperbolic domain. 
 A \textit{differential Herglotz vector field} on $D$ is a mapping $G:[0,\infty)\times D \to \C^n$ with the following properties:
\begin{itemize}
 \item[(1)] The function $t\mapsto G(t, z)$ is measurable on $[0,\infty)$ for all $z\in D$.
\item[(2)] For any compact set $K\subset D$ and $T>0,$ there exists $C_{T,K}>0$ such that $\|G(t,z)\|\leq C_{T,K}$ for all $z\in K$ and for almost all $t\in[0,T]$. 
\item[(3)] $\int_0^t G(s,z) ds$ is an infinitesimal generator on $D$ for every $t\geq 0$.
\end{itemize}
\end{definition}

\begin{theorem}\label{intro_1b} Let $D\subset \C^n$ be a domain which is bounded and convex. 
Let $G:[0,\infty)\times D\to \C^n$ be a differential Herglotz vector field on $D$ and let $K_t:D\to D$ be the nonlinear resolvent of $\int_0^t G(s,z) ds$ at time $1$, i.e.\ the inverse of $z\mapsto z- \int_0^t G(s,z) ds$. Then $t\mapsto K_t(z)$ is locally absolutely continuous for every $z\in D$, $t\mapsto K_t$ is continuous with respect to locally uniform convergence and there exists a set $N\subset[0,\infty)$ of Lebesgue measure $0$ such that
\begin{equation}\label{m2_intro}\frac{\partial}{\partial t}K_t(z)=K_t'(z)\cdot G(t,K_t(z)) \quad \text{for all \ $t\in[0,\infty)\setminus N$ and all $z\in D$}, \quad K_0(z)= z.\end{equation}
Let $(f_t)_{t\geq 0}$ be a family of holomorphic mappings $f_t:D\to D$ satisfying \eqref{m2_intro}, such that $t\mapsto f_t(z)$ is locally absolutely continuous for every $z\in D$ and $t\mapsto f_t$ is continuous with respect to locally uniformly convergence. Then $f_t=K_t$ for all $t\geq 0$.
\end{theorem}

In contrast to equation \eqref{auto}, the solution in Theorem \ref{intro_1b} is not necessarily a decreasing Loewner chain, see Example \ref{no_solution}.

%
%

Convexity of $D$ is a natural assumption for the resolvent equation. 
On the upper half-plane $\uhp$ -- which is unbounded, convex, and still complete hyperbolic -- 
the resolvent equation does not have solutions for all cases, see Example \ref{ex:2t}. The following theorem shows that this is rather an exceptional case.

\begin{theorem}\label{thm:unbdd}
Let $D$ be an unbounded convex domain in $\C$ with $D\ne\C$ and
$G:D\to\C$ be an infinitesimal generator of a semigroup $(F_t)$ on $D.$
If the Denjoy-Wolff point $\sigma$ of $(F_t)$ is finite, then
$G$ admits nonlinear resolvents $J_t$ on $D$ for all $t\ge 0.$
\end{theorem}

A definition of the Denjoy-Wolff point will be given in Section 2.
If the Denjoy-Wolff point is $\infty$, the question concerning the existence of nonlinear resolvents is more complicated. 
In Section \ref{sec_strip} we give a sufficient condition for generators on a horizontal strip having Denjoy-Wolff point $\infty$.
For the case $D=\uhp$, we obtain a complete characterization.

\begin{theorem}\label{intro_2}Let $G:\uhp\to\C$ be an infinitesimal generator on $\uhp$.
\begin{enumerate}

\item[(1)] Assume that $G(\uhp)\subset \uhp\cup \R$. (This is equivalent to either $G(z)\equiv 0$ or the associated semigroup has Denjoy-Wolff point $\infty$.) Let $b=\lim_{y\to\infty}G(iy)/(iy)$. Then
$b$ is a non-negative real number and the resolvent equation for $G$ can be solved if and only if $t\in[0,1/b)$.

If $b=0$, then the resolvents $f_t$ satisfy \[\frac1{f_t(z)}=\int_\R \frac1{z-x}\mu_t(dx),\] where $(\mu_t)_{t\geq0}$ is a semigroup of probability measures on $\R$ with respect to additive free convolution.
\item[(2)] In all other cases, the nonlinear resolvent of $G$ exists for every $t\ge0.$
\end{enumerate}
\end{theorem}

The resolvents in Theorem \ref{intro_2} also satisfy \eqref{intro_2_eq}. Note that Theorem \ref{intro_2} (2) follows directly from Theorem \ref{thm:unbdd}.\\

The paper is organized as follows. 
\begin{itemize}
	\item In Section \ref{sec_semi} we briefly recall some facts on semigroups and nonlinear resolvents.
	\item The proof of Theorem \ref{thm:unbdd} is provided in Section \ref{sec_interior}.
	\item In Section \ref{sec_deLo} we look at the general theory of decreasing Loewner chains and we prove Theorems \ref{intro_1} and \ref{intro_1b}.
	\item The proof of part (1) of Theorem \ref{intro_2} is based on notions from free probability theory, which we discuss in Section \ref{sec_free}.
\end{itemize}

\section{Semigroups and nonlinear resolvents}\label{sec_semi}

Let $D\subset\C^n$ be a domain. 
For a holomorphic map $f:D\to \C^n$ we denote by $f'(z)$ its (holomorphic) Jacobian matrix at $z$. 
For $n=1$, $f'(z)$ is the ordinary derivative of $f(z).$ 
We denote by $I:\C^n\to\C^n$ the identity mapping.

Let $(F_t)_{t\geq0}$ be a semigroup of holomorphic functions on $D$, i.e.\ $F_0$ is the identity, $F_{t+s}=F_s\circ F_t$, and 
$t\mapsto F_t$ is continuous with respect to locally uniform convergence on $D.$
In the sequel, we sometimes call it a \textit{semigroup} for short.
Then the locally uniform
limit \[\lim_{t\downarrow 0}\frac{F_t(z)-z}{t}=:G(z)\]
exists and the holomorphic function $G:D\to\C^n$ is called the \textit{infinitesimal generator}
(or, simply, the generator)
 of the semigroup, see \cite[Theorem 5]{Aba92}. 
(Sometimes, e.g. in \cite{ESS20}, the function $-G$ is called the infinitesimal generator.)
We denote by $\G(D)$ the set of infinitesimal generators on $D.$

The Berkson-Porta formula \cite{BP78} gives a form of infinitesimal generators on the unit disk $\D.$

\begin{lemma}\label{lem:BP}
Let $G\in\G(\D).$
Then, there exist a point $\tau\in \overline{\D} := \D \cup\partial \D$ and a holomorphic function $p:\D\to\C$ with $\Re(p)\geq0$ such that
\begin{equation}\label{Berkson-Porta}
 G(z) = (\tau-z)(1-\overline{\tau}z)p(z).
 \end{equation}
Conversely, any function $G(z)$ of the form \eqref{Berkson-Porta} is an infinitesimal generator on $\D.$
\end{lemma}

If $G(z)\not\equiv 0$, we call $\tau$ the \textit{Denjoy-Wolff point} of the semigroup $(F_t)_{t\geq0}$ generated by $G$. 
Unless $(F_t)_{t\geq0}$ consists of elliptic automorphisms, 
we have $F_t(z)\to \tau$ as $t\to\infty$ locally uniformly on $\D$. 
Here, we recall that an automorphism of $\D$ is called \textit{elliptic} if it has
an interior fixed point in $\D.$
If $(F_t)_{t\geq0}$ does consist of elliptic automorphisms, then its Denjoy-Wolff point is defined to be
its unique common fixed point, i.e. $\tau\in\D$ and $F_t(\tau)=\tau$ for all $t\geq0$.
See \cite{BCDM20} for details.

Let $\varphi$ be a conformal homeomorphism of $\D$ onto a simply connected domain $D$ in the Riemann sphere $\sphere=\C\cup\{\infty\}.$
It is known that $\varphi$ extends to a continuous map $\overline\D$ to
$\overline D\subset\sphere$ if and only if $\partial D$ is locally connected (see e.g. \cite[Theorem 2.1]{Pom92}).
In this case, we can generalize the notion of the Denjoy-Wolff point as follows.
Let $(F_t)_{t\ge0}$ be a semigroup on $D.$
Suppose that the conjugated semigroup
$\hat F_t=\varphi^{-1}\circ F_t\circ\varphi$ on $\D$ has the Denjoy-Wolff
point at $\tau\in\overline\D.$
We may define $\sigma=\varphi(\tau)\in\overline D$ to be the Denjoy-Wolff point
of the semigroup $(F_t)$ on $D.$
Indeed, $F_t\to \sigma$ locally uniformly on $D$ as $t\to\infty$ unless $(F_t)$ consists
of automorphisms of $D$ with a common interior fixed point at $\sigma.$

We are interested in general characterizations of infinitesimal generators.
One of them is described in terms of nonlinear resolvents as an analogue of the Hille-Yosida theory for semigroups of linear operators on a Banach space.

\begin{definition}
Let $D$ be a domain in $\C^n$ and $G:D\to \C^n$ be a holomorphic map. 
Let $t\geq 0$. If the equation
\begin{equation}\label{nl} w = z - t\cdot G(z) \end{equation}
has a solution $z$ in $D$ for each $w\in D$ in such a way that
the mapping $w\mapsto z=J_t(w)$ is a holomorphic self-map of $D,$ then
the mapping $J_t(w)=J_t(w,G)$ is 
called the \textit{nonlinear resolvent} (at ``time'' $t$) of $G$.
\end{definition}

It is noteworthy to mention the following fact.
When $D$ is a bounded convex domain in $\C^n,$ for each $w\in D$
there exists a unique solution $z=J_t(w)\in D$ to the equation $w=z-t\cdot G(z)$
as is shown at step 6 in the proof of \cite[Theorem 1.1]{RS97}.
The following characterization of infinitesimal generators is due to Reich and Shoikhet
\cite[Theorem 1.1 and Corollary 1.2]{RS97}.

\begin{theorem}[\cite{RS97}]\label{thm:RS}
Let $D$ be a bounded convex domain in $\C^n.$
A holomorphic map $G:D\to\C^n$ admits nonlinear resolvents $z=J_t(w)$
in \eqref{nl} for all $t\ge0$
if and only if $G$ is an infinitesimal generator of a semigroup $(F_t)$ on $D.$
Moreover, in this case,
$$
F_t=\lim_{n\to\infty}J_{\frac tn}\circ\cdots\circ J_{\frac tn} ~(n\operatorname{-times})
=\lim_{n\to\infty}{J_{\frac tn}}^{\circ n}.
$$
\end{theorem}


\begin{example}\label{ex:2t}
Let
$$
F_{t}(\zeta) = \frac{\zeta+\tanh(t/2)}{1+\zeta\tanh(t/2)} ,\quad \zeta\in\D,\; t\ge0.
$$
Then $(F_{t})_{t\geq0}$ is a semigroup of holomorphic mappings on $\D$ whose Denjoy-Wolff point is 1. The generator $H$ is given by $H(\zeta) = (1-\zeta^{2})/2$, and due to Theorem \ref{thm:RS} the nonlinear resolvent always exists.
On the other hand, consider the conjugate $\hat{F}_{t}$ by $C(z)=(z-i)/(z+i) :\uhp\to\D$, namely $\hat{F}_{t} := C^{-1} \circ F_{t} \circ C$. 
Then 
$$
\hat{F}_{t}(z) = e^{t}z, \quad\quad z\in\uhp,\; t\ge0,
$$
is a semigroup of holomorphic mappings on $\uhp$ with generator $G(z) = z$.
In this case the equation $w = z - t z$ has no solution in $\uhp$ for $t\geq 1$, and the resolvent $J_t(w)=\frac{w}{1-t}$ exists only for $t\in[0,1)$.
We can consider the generator $G(z)=z$ also for $D=\C$, where  $J_t$ exists if and only if $t\not=1$.
\end{example}

\begin{example}\label{ex:3t}
Consider the semigroup \[F_t(z_1,z_2)=(z_1e^{tz_1z_2}, z_2e^{-tz_1z_2})\] on $\C^2$. Its infinitesimal generator $G$ is given by $G(z_1,z_2)=(z_1^2z_2, -z_1z_2^2)$ and the nonlinear resolvent equation $(1,0)=(z_1-tz_1^2z_2, z_2+tz_1z_2^2)$ has several solutions for $t>0$, e.g.\ $(z_1,z_2)=(1,0)$ and $(z_1,z_2) = (1/2, -2/t)$. [The mapping $(z_1e^{z_1z_2}, z_2e^{-z_1z_2})$ is given in \cite{AL92} as an example of an automorphism of $\C^2$ which is not a composition of shears.]
Note that the domain $\C^2$ is not Kobayashi hyperbolic.
\end{example}

\begin{lemma}\label{rm_1}
Let $D$ be a bounded and convex domain in $\C^n$ and let $G\in \G(D)$ with nonlinear resolvents $(J_t)$. Then $G\circ J_t\in \G(D)$ for all 
$t\geq 0$. 
\end{lemma}
\begin{proof}
This is clearly true for $t=0$. So let $t>0$. Then  $(G\circ J_t)(w)= (J_t(w)-w)/t$.
Now we use the fact that $f-I$, $f:D\to D$ holomorphic, is always an infinitesimal generator on bounded convex domains, see \cite[Proposition 4.3]{RS96}, and the fact that $r\cdot G$ is an infinitesimal generator for every infinitesimal generator $G$ and $r>0$.
\end{proof}

The following lemma can be checked easily.

\begin{lemma}\label{lem:vf}
Let $(F_t)_{t\geq 0}$ be a semigroup on a domain $D\subset\C^n$ with generator $G:D\to \C^n$ and let $\varphi:\hat D\to D$ be a biholomorphic mapping for another domain $\hat D\subset \C^n$.
Then the maps $\hat{F}_t=\varphi^{-1}\circ F_t \circ \varphi$ form a semigroup on $\hat D$ with generator $\hat G(z)=(\varphi'(z))^{-1} G(\varphi(z))$.
Here, $\varphi'(z)$ is the Jacobian matrix of $\varphi$ at $z\in \hat D.$
\end{lemma}

\begin{proof}
Differentiation of both sides of $\varphi\circ\hat F_t=F_t\circ\varphi$ with respect to $t$
gives us the formula
$$
(\varphi'\circ\hat F_t)\frac{d\hat F_t}{dt}=\frac{dF_t}{dt}\circ\varphi.
$$
Letting $t=0,$ we obtain the required formula.
\end{proof}

We now use the Cayley transform $C:\uhp\to\D$, $C(z)=(z-i)/(z+i)$, to obtain a general form of generators on the upper half-plane $\uhp$. 
Let $G\in \G(\D)$. Then $G$ is expressed in the form \eqref{Berkson-Porta}.
By the last lemma, the corresponding generator $\hat G$ on $\uhp$ has the form 
\begin{align*}
\hat G(z) 
&= G(C(z))/C'(z)\\
&= G(C(z))\frac{-i (z+i)^2}{2} \\
&=
 -(z+i)^2   (\tau-C(z))(1-\overline{\tau}C(z)) \cdot \frac{i}{2}p(C(z)).
 \end{align*}
If $\tau=1$, this expression reduces to  $\hat G(z) =  2i p(C(z))$. 
If $\tau\not=1$, then the first three factors form a polynomial of degree $2$ with zeros $\sigma := C^{-1}(\tau)\in \uhp\cup \R$ and $\overline{\sigma}$. In this case we obtain
$\hat G(z)= (1-\tau)(1-\overline{\tau})(z-\sigma)(z-\overline{\sigma}) \cdot \frac{i}{2}p(C(z))$. Clearly, if $\hat G(z)\not\equiv 0$, then $\sigma$ is the Denjoy-Wolff point of the semigroup $\hat{F}_t$. Note that $ip(C(z))$ is a \textit{Pick function}, i.e. a holomorphic function from $\uhp$ into $\uhp\cup \R$. We denote by $\Nev$ the set of all Pick functions. We can summarize our calculation as follows (it is essentially contained in \cite{BP78}, where the right half-plane is considered).

\begin{lemma}\label{H_gen}
Every $G\in\G(\uhp)$ has the form 
\[  G(z)=q(z) \qquad \text{or} \qquad   G(z)=(z-\sigma)(z-\overline{\sigma})q(z), \]
with $q\in\Nev$, $\sigma \in \uhp\cup\R$. Conversely, all functions of these forms are infinitesimal generators on $\uhp$.
The first case corresponds to all semigroups whose Denjoy-Wolff point is $\infty$ unless $G=0$. In the second case, $\sigma$ is the Denjoy-Wolff point unless $G=0$.
\end{lemma}

\begin{example}
For the semigroup $F_t(z)=e^{-t}z$ on $\D$ with generator $H(z)=-z$ we obtain the resolvents $J_t(w)=w/(1+t)$ for all $t \ge 0$. 
The conjugated resolvent on $\uhp$ is given by $(C^{-1}\circ J_t \circ C)(z)=i[2z+t(z+i)]/[2i+t(z+i)]$.
On the other hand, by the above argument, the generator of the conjugated semigroup $\hat{F_{t}}$ on $\uhp$ is calculated as $G(z)=\frac{i}{2}(z^2+1) = (z-i)(z+i)\cdot\frac{i}{2}$.
Since one can observe that the equation $w = z - t\cdot G(z)$ does not lead to a M\"obius map, the nonlinear resolvents of a conjugated semigroup are in general not equal to the conjugated resolvents.
\end{example}

\section{Existence of nonlinear resolvents in unbounded domains}\label{sec_interior}

\subsection{Proof of Theorem \ref{thm:unbdd}}

We denote by $\D(a,r)$ the disk $|z-a|<r$ in the complex plane $\C.$
For $\tau\in\D$ and $0<\rho<1,$ we define
$$
\Delta(\tau, \rho)=\{z\in\D: \big|\tfrac{z-\tau}{1-\bar\tau z}\big|<\rho\}
=\D\Big(\tfrac{(1-\rho^2)\tau}{1-\rho^2|\tau|^2},\tfrac{(1-|\tau|^2)\rho}{1-\rho^2|\tau|^2}\Big),
$$
which is the hyperbolic disk of center $\tau$ with radius $\frac12\log\frac{1+\rho}{1-\rho}.$
For $\tau\in\partial\D$ and $0<R<+\infty,$ we define
$$
E(\tau,R)=\{z\in\D: \tfrac{|z-\tau|^2}{1-|z|^2}<R\}
=\D\big(\tfrac{\tau}{1+R},\tfrac{R}{1+R}\big),
$$
which is often called a horocycle of center $\tau.$
Note that
$$
\bigcup_{0<\rho<1}\Delta(a,\rho)=\D
\aand
\bigcup_{0<R<+\infty} E(\tau,R)=\D
$$
for each $a\in\D$ and $\tau\in\partial\D$.

Let $f:\D\to\D$ be a non-identity holomorphic map. We define the \textit{Denjoy-Wolff point} of $f$ similar to the definition for semigroups. If $f$ is not an elliptic automorphism, then the Denjoy-Wolff point is the unique point $\tau\in\overline{\D}$ such that $f^{\circ n}$, the $n$-th iterate of $f$, converges to $\tau$ locally uniformly on $\D$ as $n\to\infty$. If $f$ is an elliptic automorphism, we define the Denjoy-Wolff point of $f$ as the unique fixed point $\tau\in\D$ of $f$.

If $(F_t)_{t\geq0}$ is a semigroup on $\D$ with Denjoy-Wolff point $\tau$, then the Denjoy-Wolff point of each $F_t$ is equal to $\tau$ as well provided $F_t$ is not the identity.

\begin{lemma}\label{lem:SPDW}
Let $f:\D\to\D$ be a non-identity holomorphic map with Denjoy-Wolff point $\tau\in\overline{\D}$.
\begin{enumerate}
\item[(i)]
When $|\tau|<1,$ $f(\Delta(\tau,\rho))\subset\Delta(\tau,\rho)$ for all $\rho\in(0,1).$
\item[(ii)]
When $|\tau|=1,$ $f(E(\tau,R))\subset E(\tau,R)$ for all $R>0.$
\end{enumerate}
\end{lemma}

The assertion (i) follows from the Schwarz-Pick lemma and the assertion (ii)
is contained in the Denjoy-Wolff theorem (see \cite[Theorem 1.8.4]{BCDM20}).

For the proof of Theorem \ref{thm:unbdd}, we also need the following well-known result due to Study \cite{S11} and Robertson \cite{R36}.

\begin{lemma}\label{lem:convex}
Let $\varphi$ be a holomorphic function mapping $\D$ univalently onto a convex domain.
Then $\varphi(\Delta)$ is convex for every disk $\Delta$ contained in $\D.$
\end{lemma}


\begin{proof}[Proof of Theorem \ref{thm:unbdd}]
Let $D$ be an unbounded convex domain in $\C$ with $D\ne\C$ and
$G:D\to\C$ be an infinitesimal generator of a semigroup $(F_t)$ on $D.$
Furthermore, we assume that the Denjoy-Wolff point $\sigma$ of $(F_t)$ is finite. We need to show that $G$ admits nonlinear resolvents $J_t$ on $D$ for all $t\ge 0$.

First, let $\varphi:\D\to D$ be a conformal homeomorphism.
As $D$ is convex, it is a Jordan domain in the Riemann sphere unless $D$ is a parallel strip, see \cite[Corollary 2]{K87}.
Therefore, if $D$ is a Jordan domain, $\varphi$ extends to a homeomorphism $\overline\D\to \overline D$ by
a theorem of Carath\'eodory.
Even when $D$ is a parallel strip; i.e., $\varphi=A\log\frac{1+\psi}{1-\psi}+B$ for some disk automorphism $\psi$ and constants $A, B$ with $A\ne0$, $\varphi$ extends to a continuous map $\overline\D\to \overline D$.

We now consider the conjugated semigroup 
$\hat F_t=\varphi^{-1}\circ F_t\circ\varphi$
and denote by $\hat G$ its generator.
Then the Denjoy-Wolff point $\tau$ of $(\hat F_t)$ is mapped to $\sigma$ by $\varphi.$

\medskip
\noindent
Case 1: $\sigma\in D$:
Then $\tau\in\D.$
Let $D_\rho=\varphi(\Delta(\tau,\rho))$ for $0<\rho<1.$
Note that $D_\rho$ is a bounded convex domain by Lemma \ref{lem:convex}.
Then Lemma \ref{lem:SPDW} now implies that $F_t(D_\rho)\subset D_\rho$ for each $\rho$. (If $F_t$ is the identity for some $t$, then trivially $F_t(D_\rho)\subset D_\rho$.)
This enables us to regard $(F_t)$ as a semigroup on the bounded convex domain $D_\rho.$
By Theorem \ref{thm:RS}, we obtain the unique nonlinear resolvents $z=J_t(w, D_\rho)$ on 
$D_\rho$ for all $t\ge 0$ and $\rho\in(0,1);$ 
that is to say, there is a holomorphic
map $z=J_t(w, D_\rho)$ of $D_\rho$ into itself satisfying
the equation $w=z-t\cdot G(z)$ for each $w\in D_\rho.$
Since $D=\cup_{0<\rho<1}D_\rho,$ we readily see that
the maps $z=J_t(w,D_\rho)$ piece together to a holomorphic map $z=J_t(w)$ satisfying
the equation $w=z-t\cdot G(z)$ for each $w\in D.$
Thus the assertion has been shown in this case.

\medskip
\noindent
Case 2: $\sigma\in\partial D$:
In this case we have $\tau\in\partial\D$ so that we use $E(\tau, R)$ instead.
Let $D_R=\varphi(E(\tau,R))$ for $R>0.$
Then Lemma \ref{lem:SPDW} again implies that $F_t(D_R)\subset D_R$ for each $R.$
Since $\varphi(\tau)=\sigma$ is finite by assumption, $D_R$ is bounded convex
by Lemma \ref{lem:convex}.
We can now show the assertion in the same way as in Case 1.
\end{proof}

\subsection{Parallel strips}\label{sec_strip}

As another interesting example, we consider parallel strips here.
Let $D$ be the standard one defined by $|\Im z|<\pi/2.$
Though the infinite boundary point of $D$ is $\infty,$ just one point, it
represents two prime ends of $D$ as the impression.
That is, the prime end corresponding to the limit $\Re z\to +\infty$
and that corresponding to the limit $\Re z\to -\infty.$
We denote these two prime ends by $+\infty$ and $-\infty,$ respectively.
First we observe the semigroup $F_t(z)=z+t,\ z\in D, t\ge 0.$
Then the generator is $G(z)=1$ and the nonlinear resolvent exists for each $t$
and is given by $J_t(w)=w+t.$
It is the same as $F_t$ incidentally.
Note that the Denjoy-Wolff point is $+\infty.$
So far, we have a sufficient condition for existence of nonlinear resolvents
on the parallel strips when the Denjoy-Wolff point is at infinity.
More precisely, the Denjoy-Wolff point should be understood as a prime end
when the boundary of the domain is not a Jordan curve.
Without loss of generality, we may assume that the Denjoy-Wolff point is $+\infty$
in the case of the standard parallel strip.

Recall that Theorem \ref{thm:unbdd} covers the case when the Denjoy-Wolff point
is finite.

\begin{theorem}
Let $D=\{z\in\C: |\Im z|<\pi/2\}$ and $G\in\G(D).$
Suppose that the Denjoy-Wolff point $\sigma$ of the semigroup 
generated by $G$ is $+\infty.$
Then $G(z)=e^{-z}q(z)$ for a holomorphic function $q$ with $\Re q\ge0.$
Moreover, the nonlinear resolvents of $G$ exist for all 
$t\in[0,\frac{\pi}{2c}),$ where
$$
c=\inf_{x\in\R} |\Im G(x)|=\inf_{x\in\R} e^{-x}|\Im q(x)|.
$$
\end{theorem}

\begin{proof}
The function $\varphi(\zeta)=\log\frac{1+\zeta}{1-\zeta}$ maps $\D$
onto $D.$
By assumption, the point $\tau=1$ corresponds to $\sigma=+\infty$ under the mapping $\varphi.$
Then, by Lemmas \ref{lem:BP} and \ref{lem:vf}, $\hat G=G\circ\varphi/\varphi'$
has the form $(1-\zeta)^2p(\zeta),$ where $p$ is a holomorphic function on $\D$
with $\Re p\ge 0.$
Hence, for $z=\varphi(\zeta),$ we obtain
$$
G(z)=\hat G(\zeta)\varphi'(\zeta)
=2\frac{1-\zeta}{1+\zeta}\,p(\zeta)
=e^{-z}q(z).
$$
Here, we put $q(z)=2p(\varphi^{-1}(z)).$
For a while, we assume that $q$ satisfies the additional conditions
\begin{enumerate}
\item[(i)]
$q(z)$ is holomorphic on an open set
containing the closed parallel strip $|\Im z|\le \pi/2,$ and
\item[(ii)]
$q$ satisfies the inequality $|q-1|\le k|q+1|$ on $D$ for a constant
$0<k<1.$
\end{enumerate}

Condition (ii) says that the image $q(D)$ is contained in a compact
subset of the right half-plane $\Re w>0$.

Under these conditions, we will show that the nonlinear resolvents
of $G$ exist for all $t\ge 0.$
For a given point $w_0\in D$ and a positive number $t,$ we consider the function
$w=f(z)=z-tG(z)-w_0.$
We need to show that $f$ has a unique zero in $D.$
Note that $|\Im w_0|<\pi/2.$
First we observe that for $x\in\R,$
$
f(x+i\pi/2)=x+i\pi/2+it e^{-x}q(x+i\pi/2)-w_0
$
so that $\Im f(z)=\pi/2-\Im w_0+t e^{-x}\Re q(z)>0$ on $\Im z=\pi/2.$
Similarly, we see that $\Im f(z)<0$ on $\Im z=-\pi/2.$
When $x=\Re z\to+\infty$ in $D,$ we have $\Re f(z)=x+O(1).$
In particular, $\Re f(z)>0$ if $\Re z$ is large enough.
When $x=\Re z\to-\infty$ in $D,$ we observe
$$
\arg(-f(z))=\arg e^{-z}+\arg(q(z)-e^z(z-w_0)/t)
=-\Im z+\arg(q(z)+o(1)).
$$
Note that $|\arg q|\le 2\arctan k<\pi/2$ on $D$ by assumption.
We fix a number $\alpha$ so that $2\arctan k<\alpha<\pi/2.$
If $-x=-\Re z$ is positive and large enough, then
$|\arg (-f(z))|<\pi/2+\alpha<\pi.$
Now we take the boundary of the rectangle $|\Re z|\le T$
and $|\Im z|\le\pi/2$ as a contour $\Gamma$ for a large enough $T>0.$
Then we see that
$$
\frac1{2\pi i}\int_\Gamma \frac{f'(z)}{f(z)}dz
=\frac1{2\pi}\int_{f(\Gamma)} d\arg w=1
$$
as required.

Next we consider the general $G(z)=e^{-z}q(z)$ and let $0<t<\pi/(2c)$ so that $c<\pi/(2t).$
We first assume that $\Re q>0.$
Then by the definition of $c,$ we can find a real number $x_0$
so that $|\Im G(x_0)|<\pi/(2t).$
Thus $w_0=x_0-tG(x_0)\in D.$
Let $p(\zeta)=q(\varphi(\zeta)+x_0).$
Then $p$ is holomorphic and satisfies $\Re p>0$ on $\D.$
For $0<\rho<1,$ we set $G_\rho(z)=e^{-z}q_\rho(z),$
where $q_\rho(z)=p(\rho\varphi^{-1}(z-x_0)).$
Then $G_\rho$ satisfies the conditions (i) and (ii) and $G_\rho(x_0)=G(x_0)$
so that we can find the nonlinear resolvent $g_\rho(w)=J_t(w,G_\rho)$
for each $\rho.$
Note that $g_\rho:D\to D$ is holomorphic and satisfies $g_\rho(w_0)=x_0.$
Since $D$ excludes a disk, 
$(g_\rho)$ is a normal family by the Montel theorem.
Hence, there exists a locally uniform
limit $g=\lim g_{\rho_j}:D\to \overline D$ 
for a suitable sequence $\rho_j\to 1.$
Since $g(w_0)=x_0\in D,$ we see that $g$ maps $D$ into itself.
By definition of the nonlinear resolvents, the functions $g_\rho$
satisfy the relation
$$
w=g_\rho(w)-tG_\rho(g_\rho(w)),\quad w\in D.
$$
We now take the limit through the sequence $\rho_j$ to get
the relation $w=g(w)-tG(g(w)).$

Next we consider the case $\Re q=0$ at some point in $D.$
Then $q=ia$ for a real constant $a.$
In this case, we consider $q_\rho=\rho+ia$ for a positive number $\rho>0.$
Then the same argument works as $\rho\to 0.$
Hence, we have obtained the nonlinear resolvent $g(w)=J_t(w,G).$
\end{proof}

\section{Decreasing Loewner chains}\label{sec_deLo}

\begin{definition}
Let $D\subset\C^n$ be a domain. 
 A family $(f_t)_{t\geq0}$ of univalent mappings $f_t:D\to D$ is called a \textit{decreasing Loewner chain} if $f_0$ is the identity, 
 $f_t(D)\subset f_s(D)$ whenever $s\leq t$, and $t\mapsto f_t$ is continuous in the topology of locally uniform convergence on $D$. 
\end{definition}

Decreasing Loewner chains can be obtained by solving a certain Loewner PDE. First, we need the following definition, which was introduced in 
  \cite{MR2507634} for general complete (Kobayashi) hyperbolic manifolds. We will only consider complete  hyperbolic subdomains of $\C^n$. For example, every bounded and convex domain is complete hyperbolic, see \cite[Theorem 1.1]{BS09}.
	
\begin{definition}\label{def_hvf}
Let $D\subset \C^n$ be a complete hyperbolic domain. 
 A \textit{Herglotz vector field} of order $d\in[1,+\infty]$ on $D$ is a mapping $G:[0,\infty)\times D \to \C^n$ with the following properties:
\begin{itemize}
 \item[(1)] The function $t\mapsto G(t, z)$ is measurable on $[0,\infty)$ for all $z\in D.$
\item[(2)] For any compact set $K\subset D$ and $T>0,$ there exists a function $C_{T,K}\in L^d([0,T],\R^+_0)$ such that $\|G(t,z)\|\leq C_{T,K}(t)$ for all $z\in K$ and for almost all $t\in[0,T].$ 
\item[(3)] The function $z\mapsto G(t, z)$ is an infinitesimal generator for almost every $t\in[0,\infty).$
\end{itemize}
A Herglotz vector field of order $1$ will simply be called a Herglotz vector field.
\end{definition}
\begin{remark}
In \cite{MR2507634} the definition of a Herglotz vector field involves a condition on the function $G(\cdot,t)$ (for almost all $t$) using the Kobayashi metric of $D$ instead of property (3). In \cite[Theorem 1.1]{MR2887104}, however, it was shown that this condition is equivalent to condition (3).
\end{remark}

\begin{theorem}\label{thm_1}
Let $G(t,z)$ be a Herglotz vector field on a complete hyperbolic domain $D$ in $\C^n$.
Then there exists a unique solution $(f_t)_{t\geq 0}$ of holomorphic functions on $D$, such that $t\mapsto f_t(z)$ is locally absolutely continuous for every $z\in D$, to
\begin{equation}\label{inv_equ}
 \frac{\partial}{\partial t}f_t(z)=f'_t(z)\cdot G(t,z) \quad \text{for a.e.\ $t\geq 0$ and all $z\in D$}, \quad f_0(z)=z\in D.
\end{equation}
The solution $(f_t)_{t\geq 0}$ is a decreasing Loewner chain on $D$.
\end{theorem}

\begin{remark}A standard argument in Loewner theory shows that the solution $(f_t)$ satisfies \eqref{inv_equ} for all $z\in D$ and all $t\in[0,\infty)\setminus N$ for a set $N\subset [0,\infty)$ of Lebesgue measure $0$ which is independent of $z$, see \cite[Proposition 5.1]{ABHK13}.
\end{remark}


\begin{proof}We follow \cite[Theorem 3.2]{CDMG14}, where decreasing Loewner chains are handled by studying evolution families (arising from increasing Loewner chains).\\

Fix $T>0$. For $0\leq s < T$, we consider the Loewner equation
\begin{equation*}
 \frac{\partial}{\partial t}g_{s,t}(z)= G(T-t, g_{s,t}(z)) \quad \text{for a.e. $t\in[s,T]$, \quad} g_{s,s}(z)= z \in D,
\end{equation*}
which has a unique solution of univalent mappings $g_{s,t}:D\to D$ with $g_{s,s}(z)\equiv z$, 
$g_{s,t}=g_{u,t}\circ g_{s,u}$ for all $0\leq s\leq u \leq t \leq T$ and for any compact subset $K\subset D$ and for any $0<S\leq T$ there exists a non-negative function $k_{K,S}\in L^1([0,S],\R^+_0)$ such that for all $0\leq s\leq u\leq t\leq S$ and for all $z\in K,$
$$ \|g_{s,u}(z)-g_{s,t}(z)\| \leq \int_u^t k_{K,S}(\tau)\, d\tau, $$
see \cite[Def. 1.2, Prop. 3.1, Prop. 5.1]{MR2507634}. ($g_{s,t}$ is (a part of) an evolution family of order $1$.) 
In particular $g_{s,T}\circ g_{0,s} = g_{0,T}$. Differentiation with respect to $s$ gives 
\[
\left.\frac{\partial}{\partial u} g_{u,T}(g_{0,s}(z))\right\vert_{u=s} + g_{s,T}'(g_{0,s}(z)) \cdot G(T-s, g_{0,s}(z)) = 0,
\]
i.e. \[\left.\frac{\partial}{\partial u} g_{u,T}(w)\right\vert_{u=s} = -  g_{s,T}'(w) \cdot G(T-s, w)\]
for almost every $s$ and all $w\in g_{0,s}(D)$. However, the right side (and thus its integral w.r.t $s$) can be extended holomorphically to $D$ and thus 
$\frac{\partial}{\partial s} g_{s,T}(w)$ satisfies the above equation for all $w\in D$. 

Now define $f_t = g_{T-t,T}$, $0\leq t\leq T$. 
We have $f_t = g_{T-t,T} = g_{T-s, T} \circ g_{T-t,T-s}=f_s \circ g_{T-t,T-s}$ whenever $s\leq t$. Hence $(f_t)_{0\leq t\leq T}$ is (a part of) a decreasing Loewner chain. We have 
\begin{equation*}
 \frac{\partial}{\partial t}f_t(z)=  g_{T-t,T}'(z) \cdot G(t, z) =   f'_t(z) \cdot G(t, z)\quad \text{for a.e.\ $t\in[0,T]$}, \quad f_0(z)= z \in D. 
\end{equation*}
 By choosing another $\hat{T}>T$, we obtain a family $(\hat{g}_{s,t})_{0\leq s\leq t\leq\hat{T}}$ with $\hat{g}_{s,\hat{T}} 
=g_{s+T-\hat{T},\hat{T}+T-\hat{T}}=g_{s+T-\hat{T},T}$ for all $s\in [\hat{T}-T, \hat{T}]$. Hence 
$\hat{f}_t := \hat{g}_{\hat{T}-t,\hat{T}} = g_{T-t,T}=f_t$ for all $t \in [0, T]$. 

As we can choose $T>0$ arbitrarily large, we conclude that there exits a decreasing Loewner chain $(f_t)_{t\geq 0}$ satisfying \eqref{inv_equ}.\\

Finally, let $(J_t)_{t\geq0}$ be a family of holomorphic mappings on $D$, $t\mapsto J_t(z)$ locally absolutely continuous for every $z\in D$, satisfying \eqref{inv_equ}. Let $T>0$ and define 
$g_{s,t}$ as above. We have \[\frac{\partial}{\partial t}[(J_t(g_{0,T-t}))(z)] = J_t'(g_{0,T-t}(z))\cdot G(t,g_{0,T-t}(z)) - 
J_t'(g_{0,T-t}(z))\cdot G(t, g_{0,T-t}(z)) = 0\] 
for a.e.\ $t\in[0,T]$ and thus $J_t(g_{0,T-t}(z))=J_0(g_{0,T}(z))=g_{0,T}(z)$ for all $z\in D$ and $t\in[0,T]$. This implies $J_t = g_{0,T} \circ g^{-1}_{0,T-t}$ on $g_{0,T-t}(D)$. As $g_{0,T-t}(D)$ is an open set, the identity theorem implies 
$J_t = f_t$ on $D$. So the solution $(f_t)$ to \eqref{inv_equ} is unique.
\end{proof}

Now we turn our attention to nonlinear resolvents.

\begin{lemma}\label{unoml}
Let $D\subset \C^n$ be a bounded and convex domain and let $G\in\G(D)$ with nonlinear resolvents $J_t:D\to D$. Then $t\mapsto J_t$ is continuous with respect to locally uniform convergence.
\end{lemma}
\begin{proof}As $D$ is bounded, $\{J_t\}_{t\geq0}$ and $\{J'_t\}_{t\geq0}$ are both normal families and $\{J'_t\}_{t\geq0}$ is locally uniformly bounded.\\ 

 Put $\varphi_t(z)=z-tG(z)$. Let $t\geq 0$ and $K\subset D$ be compact. We need to show that $J_s$ converges uniformly on $K$ to $J_t$ as $s\to t$.\\

 The set $J_t(K)$ is a compact subset of $D$ and we find a second compact set $L\subset D$ with $K\subset L$ and $\eps>0$ such that $\varphi_s(J_t(K))\subset L$ for all $s\in[t-\eps,t+\eps]\cap[0,\infty)$ due to locally uniform convergence of $\varphi_s$ to $\varphi_t$ as $s\to t$. \\
Let $z\in K$ and $s\in[t-\eps,t+\eps]\cap[0,\infty)$ and put $w=J_t(z)$. Then $\varphi_s(w)\in L$. 
We have 
\[ J_t(z) - J_s(z)  = J_s(\varphi_s(w))  - J_s(z).\]
From $\varphi_s(w)-z = \varphi_s(w) -\varphi_t(w)$ and the fact that $\{J'_t\}_{t\geq0}$ is locally uniformly bounded, we see that $J_s$ converges to $J_t$ uniformly on $K$ as $s\to t$.
\end{proof}

The following theorem proves Theorems \ref{intro_1} and \ref{intro_1b}.

\begin{theorem}\label{result1}${}$
\begin{itemize}
\item[(1)] Let $D\subset \C^n$ be a bounded and convex domain and let $G\in\G(D)$ with nonlinear resolvents $J_t:D\to D$. Then $(J_t)_{t\geq0}$ 
is a decreasing Loewner chain satisfying the Loewner partial differential equation
\begin{equation}\label{m}\frac{\partial}{\partial t}J_t(w)=J_t'(w)\cdot G(J_t(w))\quad \text{for all $t\geq 0$}, \quad J_0(w)= w\in D.\end{equation}
The domains $J_t(D)$ contract to the zero set of $G$, i.e. $\bigcap_{t\geq 0} J_t(D) = G^{-1}(\{0\})$. (The set $G^{-1}(\{0\})$ might be empty.)
\item[(2)] Let $D\subset \C^n$ be a (possibly unbounded) convex domain and let $G\in\G(D)$. 
Furthermore, let $J_t:D\to D$ be a family of holomorphic functions satisfying \eqref{m}. Then $J_t$ are the nonlinear resolvents of $G$ on $D$.
\item[(3)] Let $D\subset \C^n$ be a bounded and convex domain and let $G:[0,\infty)\times D\to \C^n$ be a differential Herglotz vector field on $D$. Let $H_t(z):=\int_0^t G(s,z) ds$ and 
 $K_t:D\to D$ be the nonlinear resolvent of $H_t$ at time $1$, i.e.\ the inverse of $z\mapsto z- H_t(z)$. Then $t\mapsto K_t(z)$ is locally absolutely continuous for every $z\in D$, $t\mapsto K_t$ is continuous with respect to locally uniform convergence and there exists a set $N\subset[0,\infty)$ of Lebesgue measure $0$ such that
\begin{equation}\label{m2}\frac{\partial}{\partial t}K_t(z)=K_t'(z)\cdot G(t,K_t(z)) \quad \text{for all \ $t\in[0,\infty)\setminus N$ and all $z\in D$}, \quad K_0(z)= z.\end{equation}
Let $(f_t)_{t\geq 0}$ be a family of holomorphic mappings $f_t:D\to D$ satisfying \eqref{m2}, such that $t\mapsto f_t(z)$ is locally absolutely continuous for every $z\in D$ and $t\mapsto f_t$ is continuous with respect to locally uniformly convergence. Then $f_t=K_t$ for all $t\geq 0$.
\end{itemize}
\end{theorem}
\begin{proof}${}$
\begin{itemize}
\item[(1)] 
Let $J_t$ be the resolvents of $G$. Put $\varphi_t(z)=z-tG(z)$. From $z = (\varphi_t \circ J_t)(z)$ for all $z\in D$, we get by differentiation 
with respect to $z$ that 
\begin{equation}\label{sig}J'_t(z) - t G'(J_t(z)) \cdot J'_t(z)= I, \quad \text{i.e.} \quad J'_t(z) = (I-t G'(J_t(z)))^{-1}.\end{equation}

Now fix $z\in D$ and $t\geq 0$. Then

\begin{align*}  0&\,=\,\frac{(\varphi_{t+h} \circ J_{t+h})(z)-(\varphi_{t} \circ J_{t})(z)}{h} \\
&\,=\, 
\frac{(\varphi_{t+h} \circ J_{t})(z)-(\varphi_t \circ J_t)(z)}{h}
+ 
\frac{(\varphi_{t+h} \circ J_{t+h})(z)-(\varphi_{t+h} \circ J_t)(z)}{h} 
\end{align*}
for all $h\not=0$ small enough. As $h\to 0$, the first term converges to $-G(J_t(z))$. Thus also the second term converges.
It can be written as 
\[\frac{(\varphi_{t+h} \circ J_{t+h})(z)-(\varphi_{t+h} \circ J_t)(z)}{h} = \frac{J_{t+h}(z)-J_t(z)}{h} - (t+h)\frac{G(J_{t+h}(z)) - G(J_{t}(z))}{h},\]
which converges if and only if 
\[ \frac{J_{t+h}(z)-J_t(z)}{h} - t\frac{G(J_{t+h}(z)) - G(J_{t}(z))}{h} = \frac{\varphi_t(J_{t+h}(z))-\varphi_t(J_t(z))}{h}\]
converges. In other words, we know that $s\mapsto F(s):=\varphi_t(J_s(z))$ is differentiable at $s=t$. As 
$w\mapsto \varphi_t(w)$ is invertible in a neighborhood of $w=J_t(z)$, and as $s\mapsto J_s(z)$ is continuous, we have 
$J_s(z) = \varphi_t^{-1} \circ F(s)$ for all $s$ sufficiently close to $t$. Thus $s\mapsto J_s(z)$ is differentiable 
at $s=t$. We obtain 
\[\frac{\varphi_t(J_{t+h}(z))-\varphi_t(J_t(z))}{h}\to \varphi_t'(J_t(z)) \cdot \left.\frac{\partial}{\partial s} J_s(z)\right\vert_{s=t}= (I-t G'(J_t(z))) \cdot \left.\frac{\partial}{\partial s} J_s(z)\right\vert_{s=t}\]
as $h\to0$.
 Together with \eqref{sig} we obtain 

\[\frac{\partial}{\partial t}J_t(z)=J'_t(z)\cdot G(J_t(z)) \quad \text{for all $t\geq 0$.}\]

%
%

Lemma \ref{unoml} shows that $t\mapsto J_t$ is continuous with respect to locally uniform convergence. 
Due to Lemma \ref{rm_1}, we see that $(t,z)\mapsto G(J_t(z))$ is a Herglotz vector field and thus equation \eqref{m} is a Loewner partial differential equation of the form \eqref{inv_equ}. Hence, Theorem \ref{thm_1} implies that $(J_t)_{t\geq 0}$ is a decreasing Loewner chain.\\
%

 
%

Let $z\in D$. If $G(z)=0$, then $J_t(z)=z$ for all $t\geq 0$. If $G(z)\not = 0$, then there exists $T>0$ such that $z-t\cdot G(z)\not\in D$ for all 
$t\geq T$ as $D$ is bounded. Hence, $z\not\in J_t(D)$ for all $t\geq T$ and we conclude $\bigcap_{t\geq 0} J_t(D) = G^{-1}(\{0\})$.\\

\item[(2)]
Now let $J_t:D\to D$ be a family of holomorphic functions satisfying \eqref{m}, where $D$ is a convex domain. (Note that now, we do not know yet whether $(t,z)\mapsto G(J_t(z))$ is a Herglotz vector field.)
Consider the differential equation 
\begin{equation}\label{eq_7}\frac{\partial}{\partial t}\varphi_t(z) = -G(J_t(\varphi_t(z))), \quad \varphi_0(z)=z.
\end{equation}

Fix $z\in D$. We can solve this equation at least for $t$ small enough.
A small computation shows that $\frac{d}{dt}[J_t(\varphi_t(z))]=0$, i.e.\ $J_t(\varphi_t(z))$ does not depend on $t$ and $J_t(\varphi_t(z))=J_0(\varphi_0(z))=z$. 
Hence $\frac{\partial}{\partial t}\varphi_t(z) = -G(z)$. We conclude that
$t\mapsto \varphi_t(z)$ simply describes a straight line: \[\varphi_t(z)=z-tG(z).\]
In particular, we can now define $\varphi_t(z)$ for all $z\in D$ and all $t\geq0$ by $\varphi_t(z)=z-tG(z)$. Remark that $\varphi_t(z)$ might not solve \eqref{eq_7} for all $t\geq0$.\\

Let $D_t=\{z\in D\,|\, \varphi_t(z)\in D\}$. The convexity of $D$ implies that $D_t\subset D_s$  whenever $s\leq t$. \\
Let $T>0$ and assume that $D_T$ contains a non-empty open subset. We find such $T$ as $D_0=D$ and due to continuity of $t\mapsto \varphi_t(z)$.  Let $z\in D_T$. Then $z\in D_t$ for all $t\in[0,T]$ and a local solution to \eqref{eq_7} can be extended up to $t=T$:\\ 
If the solution, denote it by $\hat{\varphi}_t$, cannot be extended beyond $[0,T]$, then the maximal solution exists for $t\in[0,T_0)$ with some $T_0\leq T$, and either $\lim_{t\nearrow T_0}\hat{\varphi}_t$ does not exist, which is impossible as $\hat{\varphi}_t=z-tG(z)$, or $\lim_{t\nearrow T_0}\hat{\varphi}_t$ exists and belongs to $\partial D$, which is excluded as $z\in D_T$. Hence $J_T(\varphi_T(z))=z$ for all $z\in D_T$ and $\varphi_T$ is univalent on $D_T$. \\

Applying $\varphi_T$ gives $\varphi_T(J_T(w))=w$ for all $w\in \varphi_T(D_T)$, and $\varphi_T(D_T)$ contains a non-empty open subset. As the left side extends holomorphically to $D$, we see that $\varphi_T(J_T(w))=w$ for all $w\in D$ by the identity theorem. So $J_T$ is univalent and $J_T(D)\subset D_T$. Furthermore, $\varphi_T$ is univalent on $D_T$ and maps $J_T(D)$ onto $D$. Thus $D_T=J_T(D)$ and $J_T:D\to J_T(D)$ is the inverse of $\varphi_T|_{J_T(D)}:J_T(D)\to D$. If $z\in D\setminus J_T(D)$, then $\varphi_T(z)\not \in D$. So $J_T$ is the nonlinear resolvent of $G$ at time $T$.\\

The proof is complete in case $D_t$ contains a non-empty open subset for all $t\geq0$. So let $T_0>0$ be the infimum of all $T>0$ such that $D_T$ does not contain a non-empty open subset and assume that $T_0<\infty$.
If $D_{T_0}$ itself contains a non-empty open subset, then this is also true for $D_{t}$ with $t\in[T_0,T_0+\eps]$ for some $\eps>0$ due to continuity of $t\mapsto \varphi_t(z)$. So $D_{T_0}$ does not contain a non-empty open subset.\\
 
Choose some $z_0\in D$ and let $t\in[0,T_0)$. Then $\varphi_t(J_t(z_0))=z_0$. Let $(t_n)_n\subset [0,T_0)$ be a sequence converging to $T_0$. 
Then $w_n := J_{t_n}(z_0)$ converges to $J_{T_0}(z_0)$ as $n\to\infty$. 
Furthermore, $\varphi_{t_n}(w_n)=z_0$, i.e.\ $w_n-t_nG(w_n)=z_0$, and, by letting $n\to\infty$, $J_{T_0}(z_0)-T_0G(J_{T_0}(z_0))=z_0$.
Thus $J_{T_0}(D)\subset D_{T_0}$. If $J_{T_0}$ is univalent, then $D_{T_0}$ contains a non-empty open subset. 
So $J_{T_0}$ is not univalent and we find $z,w\in D$ with $z\not=w$ and $J_{T_0}(z)=J_{T_0}(w)=c\in D$. We have 
$\varphi_{t_n}(J_{t_n}(z)), \varphi_{t_n}(J_{t_n}(w)) \to \varphi_{T_0}(c)$ as $n\to\infty$, but 
 $\varphi_{t_n}(J_{t_n}(z))=z$ and  $\varphi_{t_n}(J_{t_n}(w))=w$ for all $n\in\N$, a contradiction.\\

\item[(3)]
Step 1 (Regularity of $K_t$): By definition,  $K_t$ is the inverse of $\varphi_t: z\mapsto z- \int_0^t G(s,z) ds$. Clearly, $t \mapsto \varphi_t(z)$ is locally absolutely continuous for any $z\in D$. Furthermore, due to property (2) of Definition \ref{def_dif}, $t\mapsto \varphi_t$ is continuous with respect to locally uniform convergence.\\

Furthermore, as $D$ is bounded, the set $\{K_\tau\,|\, \tau\geq 0\}$ is a normal family and thus $\{K'_\tau\,|\, \tau\geq 0\}$ is locally uniformly bounded.\\

Let $s,t\geq 0$ and $z\in D$. Put $w=K_t(z)$. If $s$ is close enough to $t$, then $\varphi_s(w)\in D$. 
We have 
\[ K_t(z) - K_s(z)  = K_s(\varphi_s(w)) - K_s(z).\]
From $\varphi_s(w)-z = \varphi_s(w) -\varphi_t(w)$, we see that $\tau \mapsto K_\tau(z)$ is locally absolutely continuous for any $z\in D$, and we can argue as in Lemma \ref{unoml} to show that also $t\mapsto K_t$ is continuous with respect to locally uniform convergence.\\

Next we follow \cite[Proposition 5.1]{ABHK13}.
Let $Q\subset D$ be a countable set of uniqueness for holomorphic functions on $D$. For each $z\in Q$, we find a null-set $N_z\subset[0,\infty)$ such that 
$t\mapsto\varphi_t(z)$ is differentiable for all $t\in[0,\infty)\setminus N_z$. So $t\mapsto\varphi_t(z)$ is differentiable for all $z\in Q$ and all $t\in[0,\infty)\setminus N$, where $N$ is the null-set $N=\cup_{z\in Q}N_z$.\\

Let $t\in[0,\infty)\setminus N$ and let $h_n$ be a null-sequence such that $F_n:= (\varphi_{t+h_n}-\varphi_t)/h_n$ is defined. Then the sequence $(F_n)_n$ is locally uniformly bounded due to property (2) of Definition \ref{def_dif}.
For each $z\in Q$, $F_n$ converges as $n\to \infty$. So each locally uniformly converging subsequence of $F_n$ converges on $Q$ to the same limit. Thus each locally uniformly converging subsequence of $F_n$ converges on $D$ to the same limit, which proves that $F_n$ is locally uniformly converging.\\
Consequently, $t\mapsto\varphi_t(z)$ is differentiable for all $t\in[0,\infty)\setminus N$ and all $z\in D$.\\

Step 2 (Differential equation): $K_t$ solves the equation $w = K_t(w) - \int_0^t G(s,K_t(w)) ds$. 
 We get by differentiation with respect to $w$ that
\[K'_t(w) - \int_0^t G'(s,K_t(w)) \cdot K'_t(w) ds = I,\] i.e.\ $K'_t(w) = (I-\int_0^t G'(s,K_t(w))ds)^{-1}$.
Now assume that $t\mapsto \varphi_t$ is differentiable at $t=t_0$ for all $z\in D$. Similar to the proof of (1) we see that $t\mapsto K_t(z)$ is differentiable for all $z\in D$ at $t=t_0$ with
\[\left.\frac{\partial}{\partial t}K_t(w)\right\vert_{t=t_0}=K'_{t_0}(w)\cdot G(t_0, K_{t_0}(w)).\]

Step 3 (Uniqueness): Let $(f_t)_{t\geq 0}$ be another solution of holomorphic mappings $f_t:D\to D$, such that $t\mapsto f_t(z)$ is locally absolutely continuous for every $z\in D$, $t\mapsto f_t$ is continuous with respect to locally uniformly convergence, satisfying \eqref{m2}.\\

Choose some $z_0\in D$. Then we find some open ball $B_0\subset D$ with center $z_0$, $\overline{B_0}\subset D$, and some $\eps>0$ such that 
$f_t$ is injective on $B_0$ for all $t\in[0,\eps]$. The inverse functions $g_t$ satisfy
\[ \frac{\partial}{\partial t}g_t(w) = -G(t,w)\]
for a.e.\ $t\in[0,\eps]$ and all $w\in B_1:=\cap_{t\in[0,\eps]} f_t(B_0)$, which has non-empty interior for $\eps$ small enough. 
Similar to the argument above, we can show that $t\mapsto g_t(w)$ is absolutely continuous for each $w\in B_1$ and that $t\mapsto g_t$ is continuous with respect to locally uniform convergence. This implies $g_t(w) = w-\int_0^t G(s,w) ds$ on $B_1$. For $\eps$ small enough, $B_2:=\cap_{t\in[0,\eps]}g_t(B_1)\subset B_0$ has non-empty interior  and $f_t=K_t$ on $B_2$ for all 
$t\in[0,\eps]$. The identity theorem implies $f_t=K_t$ on $D$ for all $t\in[0,\eps]$. 

Let $T>0$ the infimum of all $t$ with $f_t=K_t$ and assume $T<\infty$. Then $f_T=K_T$ due to continuity
 and a similar argument shows that $f_t = K_t$ for all $t\in[0,T+\eps]$ for some $\eps>0$, a contradiction.
\end{itemize}
\end{proof}


We can also prove a slight variation of the second part of the theorem.

\begin{lemma}\label{lemma_1}
 Let $D\subset \C^n$ be bounded and convex and let $G:D\to\C^n$ be bounded and holomorphic. Let $(J_t)_{t\geq 0}$ be a family of holomorphic self-mappings of $D$ such that \begin{equation*}\frac{\partial}{\partial t}J_t(w)=J'_t(w)\cdot G(J_t(w))\quad \text{for all $t\geq 0$}, \quad J_0(w)= w\in D.\end{equation*}
Then $G\in\G(D)$.
\end{lemma}
\begin{proof}
The proof of part (2) of Theorem \ref{result1} shows that the nonlinear resolvent equation of $G$ can be solved for all $t\geq0$. 
It follows from the boundedness of $G$ and \cite[Corollary 1.2]{RS97} that $G\in\G(D)$.
\end{proof}

For the rest of this section, we look at the unit disk. First, we construct an example of a differential Herglotz vector field $G(t,z)$ whose solution is not a decreasing Loewner chain.

\begin{example}\label{no_solution}
Let $D=\D$ and consider the two infinitesimal generators 
\[G_1(z) = -z\frac{1 + z}{1 - z}\quad \text{and} \quad G_2(z)=-z\frac{1- z}{1+ z}.\] 
We define  $G(t,z)$ by $G(t,z)=G_1(z)$, $t\leq 1$, $G(t,z)=G_2(z)$, $t>1$. 
As any combination $sG_1+tG_2$, $s,t\geq0$, is again a generator on $\D$, $G$ is a differential Herglotz vector field and we let $(J_t)_{t\geq 0}$ be the solution to \eqref{m2}.

For $t\leq 1$, $J_t$ is simply the resolvent of $G_1$. For $t=1$, the solution of $w=z+tz(1+z)/(1-z)$ is given by $J_{1}(w)=w/(w+2)$. 
For $t\geq 1$, we have \[\frac{\partial}{\partial t}J_t(w)=J_t'(z)\cdot G_2(J_t(z)),\quad  J_{1}(z)=z/(z+2).\]
Now suppose that $H:=G_2\circ J_{1}\in\G(\D)$. We have $H(z)=-J_{1}(z)(1- J_{1}(z))/(1+ J_{1}(z))$. As $H(0)=0$ and $H$ is not constant $0$, the Berkson-Porta formula \eqref{Berkson-Porta} implies 
$\Re\left(-H(z)/z\right)>0$ on $\D$. However, for $z=e^{3i}$ we obtain $\Re\left(J_{1}(z)(1- J_{1}(z))/(1+ J_{1}(z))/z\right) = \Re\left(\frac1{2 + 3 e^{3 i} + e^{6 i}}\right)=-0.47113... < 0$.\\

So $H\not\in \G(\D)$, which implies that $(J_t)_{t\geq0}$ cannot be a decreasing Loewner chain. This can be seen as follows:

If it was a decreasing Loewner chain, we can define the mappings $(J_{1}^{-1}\circ J_{t+1})_{t\geq 0}$. They satisfy the conditions of Lemma \ref{lemma_1} for the function $G_2\circ J_{1}$. Hence, $G_2\circ J_{1}$ 
would be a generator. (In order to apply Lemma \ref{lemma_1}, $G_2\circ J_{1}$ has to be bounded. But we can consider the whole construction with $G_2(z)$ being replaced by the generator $G_2(Rz)$ for
some $R\in(0,1)$ and close enough to $1$.)
\end{example}

\begin{example}\label{solution_exists} If we replace the generator $G_2$ in Example \ref{no_solution} by $G_2(z)=-z$, then $-G_2(J_1(z))/z=J_1(z)/z=1/(z+2)$ has positive real part in 
$\D$. Thus $G_2\circ J_1\in\G(\D)$ and we obtain a decreasing Loewner chain in this case.
\end{example}

\begin{remark}
Numerical calculations suggest that every Herglotz vector field which jumps at some $T>0$ from $G_1(z) = -z(1 + e^{i\alpha}z)/(1 - e^{i\alpha}z)$ to  $G_2(z)=-z(1+ e^{i\beta}z)/(1- e^{i\beta} z)$, $\alpha,\beta\in\R$, 
leads to the same situation as in Example \ref{no_solution}, i.e.\ the solution is not a decreasing Loewner chain.

It seems to be less clear to decide whether the solution to \eqref{intro_2_eq} is a decreasing Loewner chain if $G(t,z)$ varies continuously with respect to $t$. For example, 
consider the ``slit'' Herglotz vector field \[G(t,z)=-z \frac{\kappa(t)-z}{\kappa(t)+z}\] for a continuous function $\kappa:[0,\infty)\to\partial\D$. 
There is quite some literature on the question as to which conditions on $\kappa$ guarantee that \eqref{intro_1_eq} produces slit mappings, i.e., for all $t>0$, 
the image domain of $f_t$ should have the form $f_t(\D)=\D\setminus \gamma_t$ for a simple curve $\gamma_t\subset \overline{\D}$. (Roughly speaking, if $\kappa$ is smooth enough, e.g. continuously differentiable, then $f$ is a slit mapping. We refer to \cite{ZZ18} and the references therein for such results.)

We might ask: for which $\kappa$ does the solution to \eqref{intro_2_eq} form a decreasing Loewner chain?
\end{remark}

In the case of the unit disk, we can obtain some further information about the Herglotz vector field $G\circ J_t$ appearing in \eqref{m}.

\begin{theorem}
\label{expression-G-theorem}
Let $G(z) = (\tau-z)(1-\overline{\tau}z)p(z)\in \G(\D)$ with resolvents $(J_t)_{t\geq 0}$. Then the Herglotz vector field $G_t := G\circ J_t$ can be written as 
\begin{equation}
\label{expression-G-equation}G_t(z)=(\tau-z)(1-\overline{\tau}z)p_t(z) \quad \text{with} \quad 
p_t(z) = \frac{1}{t}\frac{z-J_t(z)}{(z-\tau)(1-\overline{\tau}z)}, \Re p_t \geq 0, \quad \text{for all} \quad t>0.
\end{equation}
\end{theorem}
\begin{proof}If $p(z)\equiv 0$, then $J_t(z)\equiv z$ and the statement trivially holds true. So assume that $p\not\equiv 0$. Then the Herglotz vector field $G_t=G\circ J_t$ can be written as $G_t(z)=\frac{z-J_t(z)}{t}=(z-\tau_t)(1-\overline{\tau_t}z)p_t(z)$ with $\tau_t\in \overline{\D}$ and $p_t:\D\to \C$ is holomorphic with $\Re p_t \geq 0$ for all $t > 0$.

Case 1: $\tau \in \D$. Then $G(\tau)=0$ and thus $J_t(\tau)=\tau$, which shows $G_t(\tau)=0$ and thus $\tau_t=\tau$ for all $t>0$.

Case 2: $\tau\in \partial \D$. 
It is shown in \cite[Lemma 5.1]{ESS20} (plus the text following Lemma 5.1) 
that the Denjoy-Wolff point of $J_t$ is equal to $\tau$ for each $t>0$. For the sake of completeness, we add the necessary details.

By \cite[Theorem 1]{CDMP06} and  \cite[The Grand Iteration Theorem, p. 78]{shapiro}, we know that $G(\tau)=0$ and $G'(\tau)\leq 0$ in the sense of angular limits. Thus, for $t>0$, the function $\varphi_t(z)=z-tG(z)$ satisfies $\varphi_t(\tau)=\tau$ and $\varphi'_t(\tau)\geq 1$.

As $G$ has no zeros in $\D$, $\varphi_t$, and thus also $J_t$, has no fixed points in $\D$. Hence, the Denjoy-Wolff point $\rho$ of $J_t$ lies on $\partial \D$. We have $J_t(\rho)=\rho$ and $J_t'(\rho)\in (0,1]$. This implies $\varphi_t'(\rho)=\rho$ and $\varphi_t'(\rho)\geq 1$ and thus $G(\rho)=0$ and $G'(\rho)\leq 0$. But then $\rho$ is the Denjoy-Wolff point of the semigroup associated to $G$, i.e.\ $\rho=\tau$, again, by \cite[Theorem 1]{CDMP06} and  \cite[The Grand Iteration Theorem, p. 78]{shapiro}.

We conclude $J_t(\tau)=\tau$ and $J_t'(\tau)\in (0,1]$. As $G_t(z)=(z-J_t(z))/t$, we have $G_t(\tau)=0$ and $G_t'(\tau)\leq 0$, again in the sense of angular limits. This implies that the 
Denjoy-Wolff point of the semigroup associated to $G_t$ is equal to $\tau$. We conclude $\tau_t=\tau$ for all $t>0$.
\end{proof}

\section{Upper half-plane and Denjoy-Wolff point at infinity}
\label{sec_free}

In \cite[Section 4 and Appendix A1]{HS}, the authors explain that free semigroups of probability measures lead to 
certain decreasing Loewner chains. By comparing the differential equation of these Loewner chains to \eqref{intro_2_eq}, we see that such semigroups are closely 
related to nonlinear resolvents of certain infinitesimal generators on the upper half-plane $\uhp$.

In the following we will use the Nevanlinna representation formula for Pick functions, which corresponds to the Herglotz representation for Carath\'eodory functions.

\begin{lemma}[e.g. {\cite[Theorem 1]{cau32}}]
\label{Nev-expression}
If $q\in \Nev$, then there
exist real constants $\alpha\ge0$ and $\beta$
and a finite non-negative Borel measure $\rho$ on $\R$ such that
\begin{equation}
\label{Nev-equation}
q(z)=\alpha z+\beta+\int_\R \frac{1+tz}{t-z}\rho(dt),\quad z\in\uhp.
\end{equation}
The number $\alpha$ and $\beta$ are calculated via $\alpha = \lim_{y\to\infty}q(iy)/(iy)$ and $\beta = \Re q(i)$, respectively.
Conversely, every function expressed as above is a Pick function.\\ The triple $(\alpha,\beta,\rho)$ is also called the \emph{Nevanlinna triple} of $q$.
\end{lemma}

\subsection{Additive convolution}

Let $\mu$ be a probability measure on $\R$. The Cauchy transform (or Stieltjes transform) of a probability measure $\mu$ on $\R$ is given by 
$$G_{\mu}(z):=\int_\R\frac{1}{z-t}\,\mu(\mathrm{d}t), \quad z\in\uhp=\{z\in\C\,|\, \Im(z)>0\}.$$
The $F$-transform of $\mu$ is simply defined by $F_\mu(z):=1/G_\mu(z)$ for $z\in\uhp$.

\begin{remark}
By the definition, the $F$-transform belongs to $\Nev$. Further it is known (e.g. \cite{M92}) that it has the Nevanlinna representation
$$
F_{\mu}(z)= z+\beta+\int_\R \frac{1+xz}{x-z}\rho(\mathrm{d}x), \quad z \in \uhp,
$$
with $\beta \in \R$ and $\rho$ is a finite non-negative Borel measure on $\R$.
In view of Lemma \ref{Nev-expression}, the above fact implies that $F_\mu$ satisfies $\lim_{y \to \infty}F_{\mu}(iy)/iy = 1$.
\end{remark}

$G_\mu$ maps the upper half-plane $\uhp$ into the lower half-plane. One can now show that there exist $\alpha, \beta>0$, such that $G_\mu$ is univalent within the set $\Gamma_{\alpha,\beta}=\{z\in \uhp\,|\, |\Re(z)|<\alpha \Im(z), |z|>\beta \}$. Furthermore, there exist $\gamma, \delta>0$ such that $\Delta_{\gamma,\delta} \subset G_\mu(\Gamma_{\alpha,\beta})$ with $\Delta_{\gamma,\delta}=\{1/z\,|\, z\in\uhp, |\Re(z)|<\gamma \Im(z), |z|>\delta \}$. The set $\Delta_{\gamma,\delta}$ is a wedge in the lower half-plane, symmetric with respect to the imaginary axis, and $\overline{\Delta_{\gamma,\delta}}\cap\R=\{0\}$. So we can define the compositional right inverse $G_\mu^{-1}$ on $\Delta_{\gamma,\delta}$.\\

The $R$-transform $R_\mu$ is now defined as $R_\mu(z):=G_\mu^{-1}(z)-\frac1{z}$, which is a holomorphic function defined on a wedge $\Delta_{\gamma,\delta}$, mapping it into the lower half-plane or into $\R$; see \cite[Section 5]{BV93} or \cite[Section 3.6]{MS17}.\\

 Instead of $R_\mu$, one might also regard the Voiculescu transform $\varphi_\mu(z)$ defined by $\varphi_\mu(z):=F_\mu^{-1}(z)-z=G_\mu^{-1}(1/z)-z=R_\mu(1/z)$. It maps $\Gamma_{\gamma,\delta}$ into the lower half-plane or into $\R$.

\begin{example}\label{ex_semicircle}The semicircle distribution $W(0,1)$, with mean $0$ and variance $1$, is given by the density 
\[\frac{1}{2\pi} \sqrt{4 -x^2}, \quad  x\in [-2,2].\] 
We have $G_{W(0,1)}(z)=\frac{z-\sqrt{z^2-4}}{2}$, where the branch of the square root is chosen such that $\sqrt{\cdot}$ maps $\C\setminus [-4,\infty)$ into the upper half-plane, see Exercise 5 on page 54 of \cite{MS17}.
Solving the equation $\frac{z-\sqrt{z^2-4}}{4}=w$ gives $z=w+\frac{1}{w}$ and thus \[R_{W(0,1)}(z)= z, \quad \varphi_\mu(z)=\frac1{z}.\]
Furthermore, $F_{W(0,1)}(z)=\frac{2}{z-\sqrt{z^2-4}}$, and as $F_{W(0,1)}$ extends continuously to $\uhp\cup \R$, $F_{W(0,1)}(\uhp)$ is the unbounded complement of the curve $\{F_{W(0,1)}(x)\,|\, x\in(-2,2)\}$ in $\uhp$. A small calculation yields $\Re(G_{W(0,1)}(x))=\frac{x}{2}$ and $\Im(G_{W(0,1)}(x))= \frac{-\sqrt{4-x^2}}{2}$ for $x\in[-2,2]$. Thus, for $x\in(-2,2)$, we have 
\[F_{W(0,1)}(x) = \frac1{\frac{x}{2}+i\frac{-\sqrt{4-x^2}}{2}}=
\frac{2}{x-i\sqrt{4-x^2}}=
\frac{2x+2i\sqrt{4-x^2}}{4}=\frac{x}{2}+i\sqrt{1-\left(\frac{x}{2}\right)^2},\]
and we see that the curve is the semicircle $(\partial \D)\cap \uhp$.
\end{example}

 \begin{figure}[h]
 \begin{center}
 \includegraphics[width=0.9\textwidth]{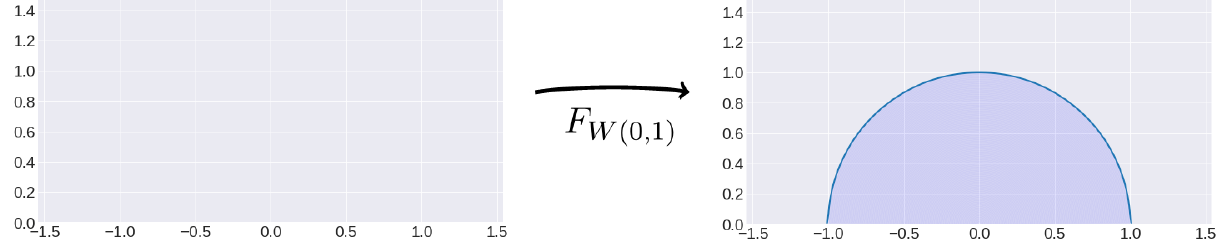}
 \caption{$F$-transform of the semicircle distribution $W(0,1)$.}
 \end{center}
 \end{figure}

For two probability measures $\mu$ and $\nu$ on $\R$, $R_\mu$ and $R_\nu$ are both defined in a sufficiently small wedge $\Delta_{\gamma,\delta}$ and it can be shown that their sum is again the $R$-transform of a probability measure.
The \textit{additive free convolution} $\boxplus$ is now defined by (see \cite{BV93} or \cite[Section 3.6]{MS17})
\begin{equation*}
R_{\mu \boxplus \nu}(z) = R_{\mu}(z)  + R_{\nu}(z) \quad \quad [\text{equivalently:}\quad  
\varphi_{\mu \boxplus \nu}(z) = \varphi_{\mu}(z)  + \varphi_{\nu}(z)].
\end{equation*}
This convolution arises from the notion of \textit{free independence} introduced by D. Voiculescu. Introductions to free probability theory can be found in \cite{ns06}, \cite{Voi97}.

A probability measure $\mu$ on $\R$ is called \textit{freely infinitely divisible} 
if for every $n\in\N$ there exists a probability measure $\mu_n$ on $\R$ such that 
$\mu = \mu_n \boxplus \cdots \boxplus \mu_n$ ($n$-fold convolution).

\begin{theorem}[Theorem 5.10 in \cite{BV93}] \label{thmBV93}
For a probability measure $\mu$ on $\R$, the following statements are equivalent.
\begin{enumerate}[\rm(1)]
\item $\mu$ is freely infinitely divisible.
\item $\mu=\mu_1$  for a $\boxplus$-semigroup $(\mu_t)_{t\geq0}$
 (i.e.\ $\mu_0=\delta_0,$ $\mu_{t+s}=\mu_t \boxplus \mu_s$ for all $s,t\geq0$ and $t\mapsto\mu_t$ is continuous with respect to weak convergence).
\item $-\varphi_\mu$ extends to a Pick function, i.e.\ a holomorphic function from $\mathbb{H}$ into $\mathbb{H} \cup \R$.
\item\label{FLK1} There exist $a \in \R$ and a finite non-negative measure $\rho$ on $\R$ such that
\begin{equation}\label{form0}
\varphi_\mu(z)=a +\int_{\mathbb{R}}\frac{1+z x}{z-x} \rho({\rm d}x) ,\qquad z\in \mathbb{H}.
\end{equation}
\end{enumerate}
Conversely, given $a\in\R$ and a finite non-negative measure $\rho$ on $\R$, there exists a unique $\boxplus$-infinitely divisible distribution $\mu$ which has the Voiculescu transform of the form \eqref{form0}.
\end{theorem}

\begin{remark}\label{rm_thmBV93_1}Note that the functions appearing in \eqref{form0} are exactly the holomorphic mappings $q:\uhp\to\uhp\cup\R$ with Nevanlinna triple $(0,a,\rho)$, i.e.\ a holomorphic function $q:\uhp\to\uhp\cup\R$ is of the form \eqref{form0} if and only if $\lim_{y\to\infty}q(iy)/y=0$.
\end{remark}

\begin{remark}\label{rm_thmBV93}From the proof of \cite[Theorem 5.10]{BV93}, it follows that a family $(\mu_t)_{t\geq0}$ of probability measures on $\R$ is a $\boxplus$-semigroup if and only if $\varphi_{\mu_t} = t \varphi_{\mu_1}$ for all $t\geq 0$.
\end{remark}


\begin{example}The Dirac measure $\mu=\delta_0$ is freely infinitely divisible as $\varphi_{\mu}(z)\equiv 0$. 
Example \ref{ex_semicircle} shows that the semicircle distribution $W(0,1)$ is also freely infinitely divisible. $W(0,1)$ is the ``normal distribution of free probability theory'', i.e.\ the limit distribution in the free central limit theorem, see \cite[Chapter 2]{MS17}.
\end{example}




\begin{theorem}
\label{add_gen_nonaut} 
Let $G:\uhp\to\uhp\cup \R$ be holomorphic with $\lim_{y\to\infty} G(iy)/y=0$. 
Then there exists a unique solution $(f_t)_{t\geq0}$ to
\begin{equation}\label{op0} \frac{\partial}{\partial t}f_t(z)=f_t'(z)\cdot G(f_t(z)),\quad t\geq 0, \quad f_0(z)=z\in \uhp.
\end{equation}
The solution can also be written as $(f_t)_{t \geq 0}=(F_{\mu_t})_{t \geq 0}$ for a free $\boxplus$-semigroup $(\mu_{t})_{t \geq 0}$ and 
$(f_t)_{t\geq0}$ is a decreasing Loewner chain.\\

Furthermore, the family $(f_t)_{t\geq0}$ are the nonlinear resolvents of $G$ and $G=-\varphi_{\mu_1}$.
\end{theorem}
\begin{proof}
Consider $H_t(z):=-tG(z)$, $t\geq 0$. This function is of the form \eqref{form0} 
and Remark \ref{rm_thmBV93} implies that we find a free $\boxplus$-semigroup $(\mu_{t})_{t \geq 0}$
 such that $\varphi_{\mu_t}(z) = H_t(z)$.
Put $f_t=F_{\mu_t}$ and $\Phi_t(z)=\varphi_{\mu_t}(z)+z$. Then $\Phi_t(f_t(z))=z$ for all $z\in\uhp$. Similar to the proof of Theorem \ref{result1} (1) we see that $t\mapsto f_t(z)$ is differentiable for all $z\in\uhp$ and
 \[0= \Phi'_t(f_t(z))  \cdot \frac{\partial}{\partial t}f_t(z)- G(f_t(z)),\]
and as $ \Phi'_t(f_t(z)) = 1/f_t'(z)$:
 \[\frac{\partial}{\partial t}f_t(z) = f'_t(z) G(f_t(z))\]
for all $t\geq 0$ and $z\in\uhp$. Clearly, $f_0$ is the identity as $\varphi_{\mu_0}=0$. Due to Lemma \ref{H_gen}, $G(f_t(z))$ is a Herglotz vector field and Theorem \ref{thm_1} implies that 
$(f_t)$ is a decreasing Loewner chain. Uniqueness of the solution follows from Theorem \ref{result1} (2), which implies that 
 the functions $(f_t)_{t\geq0}$ are the resolvents of the generator $G=-\varphi_{\mu_1}$.
\end{proof}

\begin{remark}The fact that the $F$-transforms of a $\boxplus$-semigroup $(\mu_t)_{t\geq 0}$ form a decreasing Loewner chain is also proved in 
\cite[Proposition 5.15]{Jek17}, even for the more general case of operator-valued free probability theory. Also in this higher dimensional 
case, the $F$-transforms, which are holomorphic functions on certain matricial upper half-planes, are the resolvents of $-1$ times
the Voiculescu transform of $\mu_1$ by essentially the same proof.  
\end{remark}

\begin{remark}Let $G$ be a generator of the form $-\varphi_{\mu_1}$. The nonlinear resolvents $(f_t)_{t\geq0}$ of $G$ form a decreasing Loewner chain satisfying the Loewner PDE with Herglotz vector field 
$G(t,z)=(G\circ f_t)(z)$. The semigroup $(g_t)_{t\geq0}$ associated to $G$ satisfies the Loewner PDE with Herglotz vector field $G(t,z)=G(z)$. 
While the mappings $f_t$ are the $F$-transforms of a semigroup with respect to free convolution, the mappings $g_t$ have a quite similar interpretation. For two probability measures $\mu$ and $\nu$, 
the monotone convolution $\mu\rhd \nu$ is defined by $F_{\mu\rhd \nu}(z)=(F_\mu \circ F_\nu)(z),$ $z\in \uhp$.

Now there exists a family of probability measures $(\mu_t)_{t \ge 0}$ on $\R$ such that $g_t = F_{\mu_t}$. 
As $F_{\mu_{s+t}} = F_{\mu_s} \circ F_{\mu_t}$, we see that this family is a semigroup with respect to monotone convolution, see \cite[Section 3]{FHS}, in particular \cite[Proposition 3.11]{FHS}. 
\end{remark}

A consequence of Theorem \ref{add_gen_nonaut} is the following result, which finishes the proof of part (1) of Theorem \ref{intro_2}.

\begin{corollary}
 Let $G:\uhp\to \uhp\cup\R$ be holomorphic with $b=\lim_{y\to\infty}G(iy)/(iy)$. Then the nonlinear resolvents $J_t$ of $G$ exist for all $t\in[0,1/b)$.
 If $b\not=0$, then the resolvent equation cannot be solved for $t\geq 1/b$.
\end{corollary}
\begin{proof}
 If $b=0$, then $G$ generates a free semigroup whose $F$-transforms yield the nonlinear resolvents of $G$ for all $t\geq0$ by Theorem \ref{add_gen_nonaut}.

Now assume that $b\not=0$. Then \eqref{Nev-equation} shows that $\hat{G}(z) = G(z)-bz$ also maps $\uhp$ into $\uhp\cup\R$. 
 The equation $w = z - tG(z)$ becomes $w = (1-tb)z - t\hat{G}(z)$. For $t\geq 1/b$ and $z\in\uhp$, the right side belongs to $-\uhp\cup \R$. Thus it has no solution $z\in\uhp$ for any $w\in \uhp$. 
For $t<1/b$, we can also write $w/(1-tb) = z - t\hat{G}(z)/(1-tb)$ and the previous case applied to the generator $\hat{G}(z)/(1-tb)$ shows the existence of the nonlinear resolvent $J_t$ of $G$. 
\end{proof}

\subsection{Multiplicative convolution}
For the sake of completeness we also discuss the case of multiplicative free convolution of probability measures on 
the unit circle $\T$. These measures again lead to nonlinear resolvents of certain generators, however, not on $\D$, as one might expect 
in the first place. Roughly speaking, the pullbacks of these measures to $\R$ via $x\mapsto e^{ix}$ lead again to nonlinear resolvents 
of generators on $\uhp$.

Let $\mu$ be a probability measure on $\T$. 
The moment generating function of $\mu$ is a holomorphic function on $\D$ defined by
\begin{equation*}
\psi_\mu(z):=\int_{\T}\frac{xz}{1-xz}\, \mu({\rm d}x)=\sum_{n=1}^\infty 
\left(\int_\T x^n \, \mu({\rm d}x)\right) z^n, \quad z \in \D.
\end{equation*}

The classical independence of random variables leads to the classical convolution, 
or Hadamard convolution, $\mu \star \nu$, with $\psi_{\mu \star \nu}=
\sum_{n=1}^\infty 
\left(\int_\T x^n \, \mu({\rm d}x)\right)
\left(\int_\T x^n \, \nu({\rm d}x)\right) z^n$.
Other notions of independence from non-commutative probability theory lead to further 
convolutions. First, we need to define the $\eta$-transform of $\mu$. Let 
\begin{equation*}
\eta_\mu(z):=\frac{\psi_\mu(z)}{1+\psi_\mu(z)}, \quad z \in \D.
\end{equation*}

We denote by $\cP(\T)$ the set of all 
probability measures $\mu$ on $\T$. We let $\cP_{\times}(\T)$ be the set of all $\mu\in \cP(\T)$ with $\eta_\mu'(0)\not=0$,
i.e.\ the first moment of $\mu$ is $\not=0$. Then 
we can invert $\eta_\mu$ in a neighborhood of $0$. Denote this locally defined function by 
$\eta_\mu^{-1}$.  The $\Sigma$-transform of $\mu$ is defined by 
\begin{equation*}
\Sigma_\mu(z):=\frac{1}{z}\eta_\mu^{-1}(z).
\end{equation*}

For two probability measures $\mu,\nu\in\cP_{\times}(\T)$, $\Sigma_\mu$ and $\Sigma_\nu$ are both defined in a sufficiently small neighborhood of $0$ and it can be shown that their product has again the form $\Sigma_{\alpha}$ for some $\alpha\in\cP_{\times}(\T)$. In this way, the \textit{multiplicative free convolution} $\boxtimes$ is defined, see \cite{Voi87}:

\begin{equation}\label{FM}
\Sigma_{\mu \boxtimes \nu}(z) = \Sigma_{\mu}(z) \Sigma_{\nu}(z).
\end{equation}

$\mu\in\cP_{\times}(\T)$ is called \textit{freely infinitely divisible} 
if for every $n\in\N$ there exists $\mu_n \in \cP_{\times}(\T)$ such that 
$\mu = \mu_n \boxtimes \cdots \boxtimes \mu_n$ ($n$-fold convolution).
Freely infinitely divisible distributions can be characterized in the following way.

\begin{theorem}[See Theorem 6.7 in \cite{BV92}]\label{MFID}
Let $\mu\in\cP_{\times}(\T)$. Then the following three statements are equivalent.
\begin{enumerate}[\rm(1)]
\item $\mu$ is freely infinitely divisible.
\item\label{CSUMFI} There exists a $\boxtimes$-semigroup $(\mu_t)_{t\geq 0}$ 
(i.e.\ $\mu_0=\delta_1$, $\mu_{t+s}=\mu_t \boxtimes \mu_s$ for all $s,t\geq0$ and $t\mapsto\mu_t$ is continuous with respect to weak convergence) 
such that $\mu_1 = \mu$.
\item\label{UMFI} There exists a holomorphic function $u:\D \to\C$ with $\Re(u)\geq 0$  such that $\Sigma_\mu = \exp(u)$.
\end{enumerate}
Moreover, the holomorphic function $u$ in \eqref{UMFI} can be characterized by the Herglotz representation
\begin{equation}\label{VectorFUMFI}
u(z) = -i \alpha +\int_{\T} \frac{1+ z \zeta}{1-z\zeta}\rho(\mathrm{d} \zeta),
\end{equation}
where $\alpha \in \R$ and $\rho$ is a finite non-negative Borel measure on $\T$.

Conversely, for any holomorphic function $u:\D \to\C$  with $\Re(u)\geq 0$,  the function $\exp(u)$ is the $\Sigma$-transform of some freely infinitely divisible $\mu$.
\end{theorem}

\begin{remark}\label{rm_MFID}From the proof of \cite[Theorem 6.7]{BV92}, it follows that a family $(\mu_t)_{t\geq0}$ of probability measures from $\cP_{\times}(\T)$ is a $\boxtimes$-semigroup if and only if $\Sigma_{\mu_t} = \exp(t u)$ for a holomorphic function $u:\D \to\C$ with $\Re(u)\geq 0$.
\end{remark}

\begin{remark}
 The multiplicative free convolution can also be considered for probability measures $\mu$
 with zero mean. Such a measure is freely infinitely divisible if and only if $\mu$ is the uniform distribution on 
 $\T$, i.e.\ $\psi_\mu(z)\equiv 0$, see \cite[Lemma 6.1]{BV92}.
\end{remark}

%
%
%

Now let $(\mu_t)_{t\geq0}$ be a free $\boxtimes$-semigroup of probability measures from $\cP_{\times}(\T)$ and let $M_t(z)=\int_{\T} \frac{1+xz}{1-xz}\mu_t(dx)$. Then $\Sigma_{\mu_t}(z)=\exp(tu(z))$ for a holomorphic $u:\D \to\C$  with $\Re(u)\geq 0$.\\ 

Let $F(z)=i\frac{1+z}{1-z}$. The function $F^{-1}(iM_t(e^{iz}))$ maps $\uhp$ into $\D\setminus\{0\}$. 
We can take the logarithm $J_t(z)=-i\log(F^{-1}(iM_t(e^{iz})))$, which is well-defined modulo shifts by $2k\pi$, $k\in\Z$. For $t=0$, we have $J_0(z)=z+2k\pi$, $k\in\Z$. We choose $k=0$ such that $J_0$ is the identity. Furthermore, we require continuity of $t\mapsto J_t$ and then each $J_t$ is uniquely defined.

\begin{theorem}\label{mult_gen}Let $(\mu_t)_{t\geq0}$ be a $\boxtimes$-semigroup of probability measures from $\cP_{\times}(\T)$. Let $J_t$ and $u$ be defined as above. Then $(J_t)_{t\geq 0}$ are the nonlinear resolvents of the infinitesimal generator
$G(z)=iu(-e^{iz})$ on $\uhp$ and $(J_t)_{t\geq0}$ is a decreasing Loewner chain.
\end{theorem}
\begin{proof}
Let $\eta_t=\eta_{\mu_t}$.
Then we have \[\eta_t^{-1}(z) = z\exp(t u(z)).\]
We have $(\eta_{t}^{-1}\circ\eta_t)(z)=z$ for all $z\in\D$ and similar to the proof of Theorem \ref{result1} (1) we see that $t\mapsto \eta_t(z)$ is differentiable for all $z\in\D$ and $t\geq 0$. We obtain

\begin{eqnarray*}
 &&0 = \left.\frac{\partial}{\partial s}\eta_s^{-1}(\eta_t(z))\right\vert_{s=t} + 
     (\eta_t^{-1})'(\eta_t(z))  \frac{\partial}{\partial t}\eta_t(z)  =
     \eta_t(z)u(\eta_t(z))\exp(t u(\eta_t(z))) \\&&+  (\eta_t^{-1})'(\eta_t(z))  \frac{\partial}{\partial t}\eta_t(z) 
		=zu(\eta_t(z)) +  (\eta_t^{-1})'(\eta_t(z))  \frac{\partial}{\partial t}\eta_t(z),
\end{eqnarray*}
 which yields the differential equation
\begin{equation*}
\frac{\partial}{\partial t}\eta_t(z) = -z u(\eta_t(z)) \cdot \eta'_t(z), \quad \eta_0(z)=z.
\end{equation*}

We put $M_t(z)=\frac{1+\eta_t(z)}{1-\eta_t(z)}=1+2\psi_{\mu_t}=\int_\T \frac{1+xz}{1-xz}\mu_t(dx)$. 
Then we obtain the partial differential equation
\begin{equation*}\label{addd}
\frac{\partial}{\partial t}M_t(z) = -z S(M_t(z)) \cdot M'_t(z), 
\quad M_0(z)= \frac{1+z}{1-z},
\end{equation*}
with $S(z)=u(\frac{1-z}{1+z})$. 
Note that $S$ maps the right half-plane $RH$ holomorphically into $RH\cup i\R$. 

For $\Im(z)>0$ define $V_t(z)=iM_t(e^{iz})$. Then $V_t$ maps the upper half-plane into itself and we obtain 
\begin{equation*}
 \frac{\partial V_t(z)}{\partial t}= i S(-iV_t(z)) V'_t(z),\quad V_0(z)=i\frac{1+e^{iz}}{1-e^{iz}}. 
\end{equation*}
According to \cite[Theorem 3.2]{CDMG14} (see also \cite[Theorem 4.7]{HS}), $V_t$ is a decreasing subordination chain and we can write $V_t = V_0 \circ \hat{J}_t$ for a decreasing Loewner chain $(\hat{J}_t)_{t\geq0}$ satisfying the same equation, i.e. 
\begin{equation*}
 \frac{\partial \hat{J}_t(z)}{\partial t}= i S(-iV_t(z)) \hat{J}_t'(z)=
 i S(-i (V_0 \circ \hat{J}_t)(z)) \hat{J}_t'(z)=G(\hat{J}_t(z)) \hat{J}_t'(z),\quad \hat{J}_0(z)=z \in \uhp, 
\end{equation*}
 with $G(z)=i S(-i V_0(z))=iu(-e^{iz})$, which maps $\uhp$ into $\uhp\cup \R$.  As $\hat{J}_0=J_0$ and $t\mapsto \hat{J_0}$ is continuous with respect to locally uniform convergence, we conclude that $\hat{J}_t=J_t$ for all $t\geq 0$.
Theorem \ref{result1} (2) finishes the proof.
\end{proof}

\begin{remark}
As $iu(-e^{iz})$ maps $\uhp$ into $\uhp\cup \R$ and $\lim_{y\to\infty}iu(-e^{-y})/(iy)=0$, we see that the generators from Theorem 
\ref{mult_gen} form a subset of the generators appearing in Theorem \ref{add_gen_nonaut}.
\end{remark}

\noindent
\textbf{Acknowledgements.}
The authors would like to express their sincere thanks to the referee for
careful reading of the manuscript and for offering many helpful suggestions,
which improved greatly the exposition.
They are also grateful to Mark Elin for discussions about basic properties of nonlinear resolvents.


\begin{thebibliography}{CDMG14}
\providecommand{\bysame}{\leavevmode\hbox to3em{\hrulefill}\thinspace}




\bibitem[Aba92]{Aba92} M.\ Abate, \textit{The Infinitesimal Generators of Semigroups of Holomorphic Maps}, 
Annali di Matematica pura ed applicata (IV), Vol. CLXI (1992), 167--180. 

\bibitem[AL92]{AL92} E.\ Anders\'{e}n, L.\ Lempert, \emph{On the group of holomorphic automorphisms of $\C^n$},  
Inventiones mathematicae 110 (1992),  371--388. 

\bibitem[AB11]{MR2887104}
L.\ Arosio, F.\ Bracci, \textit{Infinitesimal generators and the
  {L}oewner equation on complete hyperbolic manifolds}, Anal. Math. Phys.
  1 (2011), no.~4, 337--350.

\bibitem[ABHK13]{ABHK13} L.\ Arosio, F.\ Bracci, H.\ Hamada, G.\ Kohr, \emph{An abstract approach to Loewner chains} JAMA 119 (2013), 89--114. 


\bibitem[BV92]{BV92} H.\ Bercovici, D.\ Voiculescu, 
\textit{L\'{e}vy-Hin\v{c}in type theorems for multiplicative and additive free convolution}, 
Pacific J.\ Math. 153 (1992), no.\ 2, 217--248.	

 \bibitem[BV93]{BV93} \bysame, \textit{Free convolution of measures with unbounded support}, Indiana Univ. Math. J. 42 (1993), no. 3, 733--773.

\bibitem[BP78]{BP78} E.\ Berkson, H.\ Porta, \textit{Semigroups of analytic functions and composition operators}, Michigan Math. J. 25 (1978), no.~1, 101--115.

\bibitem[BS09]{BS09} F.\ Bracci, A.\ Saracco,  \textit{Hyperbolicity in unbounded convex domains}, Forum Math. 21 (2009), 815--825.

\bibitem[BCDM09]{MR2507634}
F.\ Bracci, M. D.\ Contreras, S.\ D{\'{\i}}az-Madrigal,
  \textit{Evolution families and the {L}oewner equation. {II}. {C}omplex
  hyperbolic manifolds}, Math. Ann. 344 (2009), no.~4, 947--962.
  
\bibitem[BCDM20]{BCDM20}
\bysame,
  \textit{Continuous semigroups of holomorphic self-maps of the unit disc}, 
 Springer Monographs in Mathematics. Springer, Cham, 2020. xxvii+566 pp.  
	

\bibitem[Cau32]{cau32} W.\ Cauer, \textit{The Poisson integral for functions with positive real part}, Bull. Amer. Math. Soc. 38 (1932), 
713--717. 

\bibitem[CDMP06]{CDMP06} M. D.\ Contreras, S.\ D\'\i az-Madrigal, C.\ Pommerenke, \textit{On  boundary  critical  points  for  semigroups of analytic functions}, Math. Scand. 98 (2006), 125--142.


M. D.\ Contreras, S.\ D\'\i az-Madrigal, P.\ Gumenyuk, \textit{Loewner chains in the unit disk}, Rev. Mat. Iberoam. 26 (2010), no.~3, 975--1012.

\bibitem[CDMG14]{CDMG14}
\bysame, \textit{Local duality in Loewner equations}, J. Nonlinear Convex Anal. 15 (2014), no. 2, 269--297.

\bibitem[dMHS18]{dMHS18} A.\ del Monaco, I.\ Hotta, S.\ Schlei{\ss}inger, \textit{Tightness results for infinite-slit limits of the chordal Loewner equation}, Comput. Methods Funct. Theory 18 (2018), no. 1, 9--33. 

\bibitem[dMS16]{dMS16}
A. del Monaco, S. Schlei{\ss}inger,
\textit{Multiple SLE and the complex Burgers equation},
Math. Nachr. 289 (2016), 2007--2018.

\bibitem[ESS20]{ESS20} M.\ Elin, D.\ Shoikhet, T.\ Sugawa, \textit{Geometric properties of the nonlinear resolvent of holomorphic generators}, J. Math. Anal. Appl. 483 (2020).
 
\bibitem[FHS20]{FHS} U.\ Franz, T.\ Hasebe, S.\ Schlei{\ss}inger,  \textit{Monotone Increment Processes, Classical Markov Processes, and Loewner Chains}, Dissertationes Math. 552 (2020), 119 pp.  

\bibitem[GKH20]{GKH} I. Graham, H. Hamada, G. Kohr, \textit{Loewner chains and nonlinear resolvents of the Carath{\'e}odory family on the unit ball in $\mathbb{C}^n$}. J. Math. Anal. Appl. 491 (2020), no. 1, 124289, 29 pp.

\bibitem[HS20]{HS} I.\ Hotta, S.\ Schlei{\ss}inger, \textit{Limits of radial multiple SLE and a Burgers-Loewner differential equation}, J. Theoret. Probab. 34 (2021), no. 2, 755--783.

\bibitem[Jek20]{Jek17}D.\ Jekel, \textit{Operator-valued chordal Loewner chains and non-commutative probability}, J. Funct. Anal. 278 (2020), no. 10, 108452, 100 pp.

\bibitem[Koe87]{K87} W. Koepf, \textit{Convex functions and the Nehari univalence criterion}, in: Laine, I., Sorvali, T., Rickman, S., Complex Analysis Joensuu 1987, Lecture Notes in Mathematics, vol 1351, Springer, Berlin, Heidelberg, 1987.

\bibitem[Maa92]{M92} H. Maassen, \textit{Addition of freely independent random variables}, 
J. Funct. Anal. 106 (1992), no. 2, 409--438.

\bibitem[MS17]{MS17}
J. A.\ Mingo, R.\ Speicher, \textit{Free Probability and Random Matrices}, Springer, Berlin, 2017. 

\bibitem[NS06]{ns06}  A.\ Nica, R.\ Speicher, \textit{Lectures on the Combinatorics of Free
Probability}, London Mathematical Society, Lecture Note Series 335, Cambridge
University Press, 2006.

\bibitem[Pom92]{Pom92}
Ch. Pommerenke, \textit{Boundary behaviour of conformal maps}, Grundlehren der
  Mathematischen Wissenschaften, vol. 299, Springer-Verlag, Berlin, 1992.

\bibitem[RS96]{RS96} S.\ Reich, D.\ Shoikhet, \textit{Generation theory for semigroups of holomorphic mappings in Banach spaces}, Abstr. Appl. Anal. 1 (1996), no. 1, 1--44. 

\bibitem[RS97]{RS97}  S.\ Reich, D.\ Shoikhet, \textit{Semigroups and generators on convex domains with the hyperbolic metric}, 
Atti della Accademia Nazionale dei Lincei. Classe di Scienze Fisiche, Matematiche e Naturali. Rendiconti Lincei. Matematica e Applicazioni, Serie 9, Vol. 8 (1997), no. 4, p. 231--250.

\bibitem[Rob36]{R36} M.\ S.\ Robertson, \textit{On the theory of univalent functions}, Ann. Math. 37 (1936), no. 2, 374-408.

\bibitem[Sha93]{shapiro} J. H.\ Shapiro, \textit{Composition Operators and Classical Function Theory}, Springer, 1993.

\bibitem[Stu11]{S11} E.\ Study, \textit{Konforme Abbildung einfach-zusammenh\"angender Bereiche}, Volume 2 of Vorlesungen \"uber ausgew\"ahlte Gegenst\"ande der Geometrie,  Leipzig, Berlin, B.G. Teubner, 1911-13.

\bibitem[Voi87]{Voi87} D.\ Voiculescu, \textit{Multiplication of certain noncommuting
random variables}, J.\ Operator Theory 18 (1987), 223--235.

\bibitem[Voi97]{Voi97} \bysame, \textit{Free Probability Theory}, Fields Inst. Commun. 12, Amer. Math. Soc., 1997.

\bibitem[ZZ18]{ZZ18} H.\ Zhang, M.\ Zinsmeister, \textit{Local Analysis of Loewner Equation}, arXiv:1804.03410.
\end{thebibliography}
\end{document}